\documentclass[10pt,reqno,a4paper,oneside,11pt]{amsart}%
\usepackage{amsfonts}
\usepackage{amssymb}
\usepackage{amsmath}
\usepackage{amsthm}
\usepackage{graphicx}
\usepackage{color}
\usepackage{array, boldline, makecell, booktabs}
\usepackage{cases}
\usepackage{epstopdf}
\usepackage{arydshln}
\usepackage{xcolor}
\usepackage{graphics}
\usepackage{graphicx}
\usepackage{bm}
\usepackage{subfigure}
\usepackage{amssymb}
\usepackage{algorithm}
\usepackage{algorithmic}
\usepackage{longtable}
\usepackage[colorlinks,bookmarksopen,bookmarksnumbered,citecolor=blue,urlcolor=red]{hyperref}
\usepackage{multirow}
\usepackage{lscape}
\usepackage{mathrsfs}
\usepackage{graphicx}
\usepackage{epsfig}
\usepackage{cite}

\providecommand{\U}[1]{\protect\rule{.1in}{.1in}}
\numberwithin{equation}{section}

\newtheorem{example}{Example}[section]

\newtheorem{proposition}{Proposition}[section]
\newtheorem{theorem}{Theorem}[section]

\newenvironment{breakablealgorithm}
{% \begin{breakablealgorithm}
	\begin{center}
		\refstepcounter{algorithm}% New algorithm
		\hrule height.8pt depth0pt \kern2pt% \@fs@pre for \@fs@ruled
		\renewcommand{\caption}[2][\relax]{% Make a new \caption
			{\raggedright\textbf{Algorithm~\thealgorithm} ##2\par}%
			\ifx\relax##1\relax % #1 is \relax
			\addcontentsline{loa}{algorithm}{\protect\numberline{\thealgorithm}##2}%
			\else % #1 is not \relax
			\addcontentsline{loa}{algorithm}{\protect\numberline{\thealgorithm}##1}%
			\fi
			\kern2pt\hrule\kern2pt
		}
	}{% \end{breakablealgorithm}
		\kern2pt\hrule\relax% \@fs@post for \@fs@ruled
	\end{center}
}
\def\argmin{\mbox{argmin}}
\def\proof{\noindent{\bf Proof}: } %????텣�?��??��???���?

\ifx\pdfoutput\relax\let\pdfoutput=\undefined\fi
\newcount\msipdfoutput
\ifx\pdfoutput\undefined\else
\ifcase\pdfoutput\else
\ifx\paperwidth\undefined\else
\ifdim\paperheight=0pt\relax\else\pdfpageheight\paperheight\fi
\ifdim\paperwidth=0pt\relax\else\pdfpagewidth\paperwidth\fi
\fi\fi\fi
\begin{document}
\pagestyle{myheadings}

\begin{center}
{\LARGE \textbf{Gl-QFOM and Gl-QGMRES: two efficient algorithms for quaternion linear systems with multiple right-hand sides}}\footnote{This work was funded by National Natural Science Foundation of China [grant number
11971294], the Hainan Provincial Natural Science Foundation of China [grant numbers 122QN214 and
122MS001], the Academic Programs project of Hainan University [grant numbers KYQD(ZR)-211151 and KYQD(ZR)-21119].
\par
$^{\dagger}$ Corresponding author. \par  Email address: wqw@t.shu.edu.cn, (Q.W. Wang); tli@hainanu.edu.cn (T. Li); 995272@hainanu.edu.cn (X.F. Zhang) }

\bigskip

{ \textbf{Tao Li$^{a}$,  Qing-Wen Wang$^{b,\dagger}$ Xin-Fang Zhang$^{a}$ }}
{\small
\vspace{0.25cm}

$a.$ Department of Mathematics, Hainan University, Haikou 570228, P. R. China.\\

$b.$ Department of Mathematics, Shanghai University, Shanghai 200444, P. R. China}

\end{center}
%The proposed methods save three-quarters of the theoretical costs compared to the traditional Gl-FOM and Gl-GMRES iterations for the real representations of quaternion linear systems.

\begin{quotation}
\noindent\textbf{Abstract:}
In this paper, we propose the global quaternion full orthogonalization (Gl-QFOM) and global quaternion generalized minimum residual (Gl-QGMRES) methods, which are built upon global orthogonal and oblique projections onto a quaternion matrix Krylov subspace, for solving quaternion linear systems with multiple right-hand sides. We first develop the global quaternion Arnoldi procedure to preserve the quaternion Hessenberg form during the iterations. We then establish the convergence analysis of the proposed methods, and show how to apply them to solve the Sylvester quaternion matrix equation. Numerical examples are provided to illustrate the effectiveness of our methods compared with the traditional Gl-FOM and Gl-GMRES iterations for the real representations of the original linear systems.

\vspace{3mm}

\end{quotation}

\section{\textbf{Introduction}}

 Let $\mathbb{R}$, $\mathbb{Q}^{n}$ and $\mathbb{Q}^{n\times m}$ denote the real number field, the collections of all $n$-dimensional vectors and all $n\times m$  matrices over the quaternion skew-field \cite{Hamilton,Zhang1,Rodman}
$$\mathbb{Q}=\{q_0+q_1\mathbf{i}+q_2\mathbf{j}+q_3\mathbf{k}|\mathbf{i}^2=\mathbf{j}^2=\mathbf{k}^2=\mathbf{i}\mathbf{j}\mathbf{k}=-1,q_0,q_1,q_2,q_3\in\mathbb{R}\},$$ respectively. 
Note that $\mathbb{Q}$ is an associative but non-commutative four-dimensional algebra over $\mathbb{R}$. The real counterpart of a quaternion matrix 
$\mathbf{W}=W_0+W_1\mathbf{i}+W_2\mathbf{j}+W_3\mathbf{k}\in\mathbb{Q}^{n\times m}$ with $W_0,W_1,W_2,W_3\in \mathbb{R}^{n\times m}$, defined by the linear homeomorphic mapping $\mathcal{R}(\cdot)$, is of the form
\begin{equation}\label{1-0}
\mathcal{R}(\mathbf{W}):=\begin{bmatrix}
W_0&-W_1&-W_2&-W_3\\
W_1&W_0&-W_3&W_2\\
W_2&W_3&W_0&-W_1\\
W_3&-W_2&W_1&W_0\\
\end{bmatrix}\in \mathbb{R}^{4n\times 4m},
\end{equation}
and its first block column is of the form
$$
\mathcal{R}(\mathbf{W})_{c}:=\begin{bmatrix}
W_0^T&W_1^T&W_2^T&W_3^T
\end{bmatrix}^T\in \mathbb{R}^{4n\times m}.
$$
Denote the $j$-th column of $\mathbf{W}$ by $\mathbf{W}_{:,j}, j=1,\cdots,m$. 

In this paper, we consider a few quaternion linear systems with the same coefficient matrix but different right-hand sides, as follows
\begin{equation}\label{1-1}
\mathbf{A}\mathbf{X}=\mathbf{B}, 
\end{equation}
where the matrix $\mathbf{A}\in\mathbb{Q}^{n\times n}$ is invertible and non-Hermitian, $\mathbf{B}=[\mathbf{B}_{:,1},\mathbf{B}_{:,2},\cdots,\mathbf{B}_{:,m}]$ and $\mathbf{X}=[\mathbf{X}_{:,1},\mathbf{X}_{:,2},\cdots,\mathbf{X}_{:,m}]$ are $m\times n$ rectangular quaternion matrices with the moderate size $m<<n$.
The matrix equation \eqref{1-1} has attracted a great deal of attention because it plays an increasingly important role in statistics \cite{Dyson,Ginzberg}, quantum physics\cite{Adler}, tensor computations \cite{Chu0,Chu1,Chu2}, computer science\cite{Wang1,He0,Yuan}, especially in colour image processing\cite{Jia0,He1,Chen00,Ji}.
For example, if the blurring matrix $\mathbf{A}$ on a colour image, which is represented by a pure quaternion matrix $\mathbf{X}$, such that the observed image is $\mathbf{B}$. The above process can be mathematically written as Eq.\eqref{1-1}. From a practical viewpoint, it is costly to compute its exact solution if the dimension is large. This paper aims to develop efficient structure-preserving methods, including the global quaternion full orthogonalization (Gl-QFOM) and global quaternion generalized minimum residual (Gl-QGMRES) methods,  for solving the original quaternion matrix equation. The methods save three-quarters of the theoretical costs compared with the traditional Gl-FOM and Gl-GMRES methods using the real counterpart.

Almost existing iterative methods for solving Eq.\eqref{1-1} are built upon the real or complex counterpart such that the dimension is four or two times the original. More precisely, by the linear homeomorphic mapping $\mathcal{R}(\cdot)$, Eq.\eqref{1-1} can be rewritten equivalently as the following real matrix equation
\begin{equation}\label{1-2}
\mathcal{R}(\mathbf{A})\mathcal{R}(\mathbf{X})=\mathcal{R}(\mathbf{B})\in\mathbb{R}^{4n\times 4m}.
\end{equation}
It is frequently solved by some block or global Krylov subspace methods over $\mathbb{R}$, such as the block biconjugate gradient method \cite{Leary}, the block generalized minimal residual method \cite{Simoncini}, the block quasi-minimum residual method \cite{Freund}, the global full orthogonalization method (Gl-FOM) and the global generalized minimal residual (Gl-GMRES) method \cite{Jbilou,Bouyouli}. Especially, Gl-FOM and Gl-GMRES, based on the global real Arnoldi method (Gl-RAM), were commonly utilized in the literature since they converge much faster than some existing solvers. However, the arithmetic operations and storage required by the two algorithms for Eq.\eqref{1-1} become prohibitively large with the dimension increasing. Moreover, an upper Hessenberg matrix generated by the mentioned algorithms does not inherit the JRS-symmetry of the real counterpart. To the best of our knowledge, only Jia \cite{Jia4,Chen0} presented the quaternion generalized minimal residual (QGMRES), and Li \cite{Li0} introduced the quaternion full orthogonalization (QFOM) methods for solving a system of linear equations, i.e., Eq.\eqref{1-1} with $n=1.$ Evidently, neither of these two methods can solve the original matrix equation. Inspired by these issues, we attempt to derive the structure-preserving Gl-QFOM and Gl-QGMRES methods, which inherit the algebraic symmetry of $\mathcal{R}(\mathbf{A})$ during the iterations, for solving \eqref{1-1}. The proposed methods are built upon the global quaternion Arnoldi process, which requires fewer arithmetic operations and storage than Gl-RAM. When solving the quaternion linear systems, such methods reduce to the competitive QGMRES and QFOM methods. We also apply the proposed methods for solving the well-known Sylvester quaternion matrix equation $\mathbf{A}\mathbf{X}+\mathbf{X}\mathbf{B}=\mathbf{C}$.

The remainder of this paper is as follows. In Section \ref{section-2}, we briefly recall some results associated with quaternion matrices. In Section \ref{section-3}, we develop the structure-preserving global quaternion Arnoldi method (Gl-QAM) generated by a quaternion matrix Krylov subspace. On this basis, we propose a modified global quaternion Arnoldi process, which is comparably robust to Gl-QAM. In Section \ref{section-4}, we introduce the Gl-QFOM and Gl-QGMRES methods for solving Eq.\eqref{1-1}, and then apply them for solving the Sylvester quaternion matrix equation. The convergence of our methods is also established. In Section \ref{section-5}, we record some numerical results on the randomly generated data and colour image deblurring problems, including their
comparisons with some existing methods, to show the effectiveness of the proposed methods. At last, we conclude this paper by drawing some remarks in Section \ref{section-6}.

\section{Preliminaries}\label{section-2}
In this section, we introduce some relevant preliminaries to quaternion matrices, which are useful in the later sections.
The conjugate of a quaternion $\bm{q}$ is defined by $\overline{\bm{q}}=q_0-q_1\mathbf{i}-q_2\mathbf{j}-q_3\mathbf{k}$, and its magnitude is of the form $|\bm{q}|=\sqrt{\bm{q}\overline{\bm{q}}}=\sqrt{q_0^2+q_1^2+q_2^2+q_3^2}.$ Each nonzero quaternion $\bm{q}$ has a unique inverse $\overline{\bm{q}}/{|\bm{q}|^2}$, denoted by $\bm{q}^{-1}$. The conjugate transpose of a matrix $\mathbf{X}=X_0+X_1\mathbf{i}+X_2\mathbf{j}+X_3\mathbf{k}\in \mathbb{Q}^{n\times m}$, defined by $\mathbf{X}^*=X_0^T-X_1^T\mathbf{i}-X_2^T\mathbf{j}-X_3^T\mathbf{k}$, is an $m\times n$ quaternion matrix.
If a square quaternion matrix $\mathbf{X}$ yields $\mathbf{X}=\mathbf{X}^*$, i.e., $X_0=X_0^T, X_i=-X_i^T, i=1,2,3$,  then it is said to be Hermitian; otherwise, non-Hermitian.
The inner product of two quaternion matrices $\mathbf{X},\mathbf{Y}\in \mathbb{Q}^{n\times m}$, defined as $\langle\mathbf{X},\mathbf{Y}\rangle=\mbox{tr}(\mathbf{Y}^*\mathbf{X})$, is a quaternion scalar, in which $\mbox{tr}(\mathbf{Z})$ stands for the trace of the square matrix $\mathbf{Z}$, and they are said to be orthogonal if $\langle\mathbf{X},\mathbf{Y}\rangle=0$. The induced norm of $\mathbf{X}$ is the well-known Frobenius norm defined by $\|\mathbf{X}\|=\mbox{tr}(\mathbf{X}^*\mathbf{X})$. 
As seen in \cite{Wei0, Jia1, Jia2, Jia3}, all real counterparts of a quaternion matrix are permutationally equivalent and surely JRS-symmetric matrices. If $M,N,P$ are JRS-symmetric matrices and $\alpha,\beta\in\mathbb{R}$, then so is $\alpha M+\beta NP$. Denote the inverse mapping of $\mathcal{R}$ to the JRS-symmetric matrix by $\mathcal{R}^{-1}(\mathcal{R}(\mathbf{X}))=\mathbf{X}.$ Moreover,  $\mathcal{R}(\mathbf{X})\in\mathbb{R}^{4n\times 4n}$ is called an upper JRS-Hessenberg matrix, i.e, $\mathbf{X}\in\mathbb{Q}^{n\times n}$ is called an upper Hessenberg quaternion matrix, if ${X}_0\in\mathbb{R}^{n\times n}$ is of upper Hessenberg form, and ${X}_1,{X}_2,{X}_3\in\mathbb{R}^{n\times n}$ are of upper triangular forms. 
Denote the right-hand-side linear combination of quaternion matrices $\mathbf{V}_1,\mathbf{V}_2,\cdots,\mathbf{V}_k\in\mathbb{Q}^{n\times m}$ by $\mathbf{V}_1\bm{\alpha}_1+\mathbf{V}_2\bm{\alpha}_2+\cdots+\mathbf{V}_k\bm{\alpha}_k$, in which $\bm{\alpha}_1,\bm{\alpha}_2,\cdots,\bm{\alpha}_k$ are quaternion scalars.  Throughout the paper, we focus on the right-hand-side linear combination of quaternion matrices since its properties are similar to the linear combination of complex matrices. We call that $\mathbf{V}_1,\mathbf{V}_2,\cdots,\mathbf{V}_k$ are linearly dependent if there exist $\bm{\alpha}_1,\bm{\alpha}_2,\cdots,\bm{\alpha}_k$, not all zero, such that $\mathbf{V}_1\bm{\alpha}_1+\mathbf{V}_2\bm{\alpha}_2+\cdots+\mathbf{V}_k\bm{\alpha}_k=\bm{0},$ conversely, linearly independent. As shown in \cite{Ghiloni}, the collection $\mathbb{Q}^{n\times m}$ is a right quaternionic Hilbert space with the inner product yielding the following properties:
 \begin{itemize}
\item (Right linearity) $\langle\mathbf{X}\bm{\alpha}+\mathbf{Y}\bm{\bm{\beta}},\mathbf{Z} \rangle=\langle\mathbf{X},\mathbf{Z} \rangle\bm{\alpha}+\langle\mathbf{Y},\mathbf{Z} \rangle\bm{\bm{\beta}}$ holds for $\mathbf{X},\mathbf{Y},\mathbf{Z}\in \mathbb{Q}^{n\times m}$ and $\bm{\alpha},\bm{\bm{\beta}}\in \mathbb{Q}$;
\item (Quaternionic hermiticity) $\langle\mathbf{X},\mathbf{Y} \rangle=\overline{\langle\mathbf{Y},\mathbf{X}\rangle}$ holds for $\mathbf{X},\mathbf{Y}\in \mathbb{Q}^{n\times m}$;
\item (Positivity) $\langle\mathbf{X},\mathbf{X} \rangle\geq 0$ holds for $\mathbf{X} \in \mathbb{Q}^{n\times m}$, and $\mathbf{X}=\mathbf{0}$ if and only if $\langle\mathbf{X},\mathbf{X} \rangle=0;$
\item The distance between two quaternion matrices $\mathbf{X}, \mathbf{Y}\in \mathbb{Q}^{n\times m}$ is defined by $$D(\mathbf{X},\mathbf{Y}):=\sqrt{\langle\mathbf{X}-\mathbf{Y},\mathbf{X} -\mathbf{Y}\rangle}.$$
\end{itemize}

Denote an $n\times mk$ quaternion matrix $[\mathbf{V}_1,\mathbf{V}_2,\cdots,\mathbf{V}_k]$ by $\mathbf{\mathcal{V}}_k$ with $\mathbf{V}_i\in\mathbb{Q}^{n\times m}$, $i=1,\cdots, k$. The symbol $*$ denotes the following quaternion matrix-vector product
\begin{equation}\label{2-1}
\mathbf{\mathcal{V}}_k*\bm{\alpha}:=\sum_{i=1}^k\mathbf{{V}}_i\bm{\alpha}_i,
\end{equation}
where $\bm{\alpha}=(\bm{\alpha}_1,\bm{\alpha}_2,\cdots,\bm{\alpha}_k)^T\in\mathbb{Q}^{k}$, and the quaternion matrix-matrix product
\begin{equation}\label{2-2}
\mathbf{\mathcal{V}}_k*\mathbf{W}:=[\mathbf{\mathcal{V}}_k\mathbf{W}_{:,1}, \mathbf{\mathcal{V}}_k\mathbf{W}_{:,2},\cdots,\mathbf{\mathcal{V}}_k\mathbf{W}_{:,k}]
\end{equation}
with $\mathbf{W}\in\mathbb{Q}^{k\times k}$.
From the above definitions, for any vectors $\bm{\alpha},\bm{\beta}\in\mathbb{Q}^{k}$,  the following operations with $*$ are trivial:
\begin{equation}\label{2-3}
\mathbf{\mathcal{V}}_k*(\bm{\alpha}+\bm{\beta})=\mathbf{\mathcal{V}}_k*\bm{\alpha}+\mathbf{\mathcal{V}}_k*\bm{\beta} \quad \mbox{and} \quad (\mathbf{\mathcal{V}}_k*\mathbf{W})*\bm{\alpha}=\mathbf{\mathcal{V}}_k*(\mathbf{W}\bm{\alpha}).
\end{equation}
For $\mathbf{W}\in \mathbb{Q}^{ n\times n}$ and $\mathbf{\mathcal{V}}_k=[\mathbf{V}_1,\mathbf{V}_2,\cdots,\mathbf{V}_k]\in \mathbb{Q}^{n\times mk}$, their $\boxtimes$ product, defined as
\begin{equation}\label{2-4}
\mathbf{\mathcal{V}}_k^*\boxtimes(\mathbf{W}\mathbf{\mathcal{V}}_k)=\mathbf{\mathcal{V}}_k^*\boxtimes[\mathbf{W}\mathbf{{V}}_1,\mathbf{W}\mathbf{{V}}_2,\cdots,\mathbf{W}\mathbf{{V}}_k]=:\begin{bmatrix}
\mbox{tr}(\mathbf{{V}}_1^*\mathbf{W}\mathbf{{V}}_1)&\mbox{tr}(\mathbf{{V}}_1^*\mathbf{W}\mathbf{{V}}_2)&\cdots&\mbox{tr}(\mathbf{{V}}_1^*\mathbf{W}\mathbf{{V}}_k)\\
\mbox{tr}(\mathbf{{V}}_2^*\mathbf{W}\mathbf{{V}}_1)&\mbox{tr}(\mathbf{{V}}_2^*\mathbf{W}\mathbf{{V}}_2)&\cdots&\mbox{tr}(\mathbf{{V}}_2^*\mathbf{W}\mathbf{{V}}_k)\\
\vdots&\vdots&\ddots&\vdots\\
\mbox{tr}(\mathbf{{V}}_k^*\mathbf{W}\mathbf{{V}}_1)&\mbox{tr}(\mathbf{{V}}_k^*\mathbf{W}\mathbf{{V}}_2)&\cdots&\mbox{tr}(\mathbf{{V}}_k^*\mathbf{W}\mathbf{{V}}_k)\\
\end{bmatrix},
\end{equation}
is a $k\times k$ quaternion matrix. Let $\mathbf{V}_j={V}_0^{(j)}+{V}_1^{(j)}\mathbf{i}+{V}_2^{(j)}\mathbf{j}+{V}_3^{(j)}\mathbf{k}$ with $ j=1,\cdots,k,$ and let $$
\mathcal{R}(\mathbf{\mathcal{V}}_k):=\begin{bmatrix}
\overline{\mathbf{\mathcal{V}}}_0&-\overline{\mathbf{\mathcal{V}}}_1&-\overline{\mathbf{\mathcal{V}}}_2&-\overline{\mathbf{\mathcal{V}}}_3\\
\overline{\mathbf{\mathcal{V}}}_1&\overline{\mathbf{\mathcal{V}}}_0&-\overline{\mathbf{\mathcal{V}}}_3&\overline{\mathbf{\mathcal{V}}}_2\\
\overline{\mathbf{\mathcal{V}}}_2&\overline{\mathbf{\mathcal{V}}}_3&\overline{\mathbf{\mathcal{V}}}_0&-\overline{\mathbf{\mathcal{V}}}_1\\
\overline{\mathbf{\mathcal{V}}}_3&-\overline{\mathbf{\mathcal{V}}}_2&\overline{\mathbf{\mathcal{V}}}_1&\overline{\mathbf{\mathcal{V}}}_0\\
\end{bmatrix}$$
with $\overline{\mathbf{\mathcal{V}}}_i=[{V}_i^{(1)}, {V}_i^{(2)}, \cdots,  {V}_i^{(k)}]\in \mathbb{R}^{n\times mk}$, $i=0,1,2,3$. Denote by $\hat{\overline{\mathbf{\mathcal{V}}}}_t\in \mathbb{R}^{4n\times mk}$, $t=1,2,3,4$, the $t$-th block column matrices of $\mathcal{R}(\mathbf{\mathcal{V}}_k)$. Applying the liner homeomorphic mapping $\mathcal{R}(\cdot)$ on both sides of \eqref{2-4}, it follows that
\begin{equation}\label{2-5}
\mathcal{R}(\mathbf{\mathcal{V}}_k)^T\boxtimes(\mathcal{R}(\mathbf{W})\mathcal{R}(\mathbf{\mathcal{V}}_k)):=\begin{bmatrix}
H_0&-H_1&-H_2&-H_3\\
H_1&H_0&-H_3&H_2\\
H_2&H_3&H_0&-H_1\\
H_3&-H_2&H_1&H_0\\
\end{bmatrix}\in\mathbb{R}^{4k\times 4k},
\end{equation}
where
$$
H_{t-1}=(h_{pq})=\hat{\overline{\mathbf{\mathcal{V}}}}_t^T\boxtimes[\mathcal{R}(\mathbf{W})\mathcal{R}(\mathbf{\mathcal{V}}_k)_c]\in \mathbb{R}^{k\times k}
$$
with its entries $h_{pq}=\mbox{tr}((\hat{\overline{\mathbf{\mathcal{V}}}}_t)_p^T\mathcal{R}(\mathbf{W})(\mathcal{R}(\mathbf{\mathcal{V}}_k)_c)_q)$, in which $(\hat{\overline{\mathbf{\mathcal{V}}}}_t)_p\in \mathbb{R}^{4n\times m}$ and $(\mathcal{R}(\mathbf{\mathcal{V}}_k)_c)_q\in \mathbb{R}^{4n\times m}$ are the $p$-th and $q$-th block column matrix of $\hat{\overline{\mathbf{\mathcal{V}}}}_t$ and $\mathcal{R}(\mathbf{\mathcal{V}}_k)_c$, respectively, $p,q=1,\cdots,k.$

From \cite{Jia1,Jia4}, it was shown that two types of orthogonally JRS-symplectic equivalence transformations could preserve the JRS-symmetric structures. One is the generalized symplectic Givens rotation, defined as
\begin{equation}\label{2-6}
G^{(l)}:=\begin{bmatrix}
G^{(l)}_0&-G^{(l)}_1&-G^{(l)}_2&-G^{(l)}_3\\
G^{(l)}_1&G^{(l)}_0&-G^{(l)}_3&G^{(l)}_2\\
G^{(l)}_2&G^{(l)}_3&G^{(l)}_0&-G^{(l)}_1\\
G^{(l)}_3&-G^{(l)}_2&-G^{(l)}_1&G^{(l)}_0\\
\end{bmatrix}\in\mathbb{R}^{4n\times 4n},
\end{equation}
where $$\begin{aligned}&G^{(l)}_0=\begin{bmatrix}
I^{(l-1)\times (l-1)}&0&0\\
0&\alpha_0^{(l)}&0\\
0&0&I^{(n-l)\times (n-l)}
\end{bmatrix},\quad G^{(l)}_1=\begin{bmatrix}
0^{(l-1)\times (l-1)}&0&0\\
0&\alpha_1^{(l)}&0\\
0&0&0^{(n-l)\times (n-l)}
\end{bmatrix}, \\
&G^{(l)}_2=\begin{bmatrix}
0^{(l-1)\times (l-1)}&0&0\\
0&\alpha_2^{(l)}&0\\
0&0&0^{(n-l)\times (n-l)}
\end{bmatrix},\quad G^{(l)}_3=\begin{bmatrix}
0^{(l-1)\times (l-1)}&0&0\\
0&\alpha_3^{(l)}&0\\
0&0&0^{(n-l)\times (n-l)}
\end{bmatrix}
\end{aligned}$$
for $1\leq l\leq n$ are $n\times n$ matrices, and $(\alpha_0^{(l)})^{2}+(\alpha_1^{(l)})^{2}+(\alpha_2^{(l)})^{2}+(\alpha_3^{(l)})^{2}=1$ with $\alpha_i^{(l)}\in[-1,1], 1\leq i\leq 4$. % Another is the direct sum of four identical $n\times n$ Householder matrices defined as
%$$[H^{(l)}\oplus H^{(l)}\oplus H^{(l)}\oplus H^{(l)}]%[v,\beta],$$
%where the first $l-1$ entries of the vector $v$ of length $n$ equal to $0$, and $\beta$ is a scalar satisfying $\beta(\beta v^Tv-2)=0.$ 
From the above results, for any JRS-symmetric matrix, we will prove that it can become an upper JRS-Hessenberg matrix by a JRS-symplectic matrix with orthogonal block column matrices. If the statement holds, then the global quaternion Arnoldi process (Gl-QAP) is surely true.

\section{Global Quaternion Arnoldi Method}\label{section-3}
Our goal in this section is to define the quaternion matrix Krylov subspace, and then give the structure-preserving global quaternion Arnoldi method. This method differs entirely from the classical one over the real/complex number fields since quaternions are multiplicatively noncommutative.  First, we give some practical results which will be used in the sequel.
\begin{theorem}\label{tho-1}
Let $\mathcal{R}(\mathbf{W})\in\mathbb{R}^{4n\times 4n}$ defined in \eqref{1-0} be a JRS-symmetric matrix. Then there exists a JRS-symplectic matrix $\mathcal{R}(\mathbf{V})\in\mathbb{R}^{4n\times 4mk}$ with orthogonal block column matrices such that $\mathcal{R}(\mathbf{V})^T\boxtimes (\mathcal{R}(\mathbf{W})\mathcal{R}(\mathbf{V}))=\overline{H}\in\mathbb{R}^{4k\times 4k}$ is of upper JRS-Hessenberg form.
\end{theorem}
\begin{proof}
Denote
$$\begin{aligned}
&{W}_0=\begin{bmatrix}
{w}_{11}^{(0)}&w_{12}^{(0)}&w_{13}^{(0)}&W_{14}^{(0)}\\
w_{21}^{(0)}&w_{22}^{(0)}&w_{23}^{(0)}&W_{24}^{(0)}\\
w_{31}^{(0)}&w_{32}^{(0)}&w_{33}^{(0)}&W_{34}^{(0)}\\
W_{41}^{(0)}&W_{42}^{(0)}&W_{43}^{(0)}&W_{44}^{(0)}\\
\end{bmatrix},\quad W_1=\begin{bmatrix}
{w}_{11}^{(1)}&w_{12}^{(1)}&w_{13}^{(1)}&W_{14}^{(1)}\\
w_{21}^{(1)}&w_{22}^{(1)}&w_{23}^{(1)}&W_{24}^{(1)}\\
w_{31}^{(1)}&w_{32}^{(1)}&w_{33}^{(1)}&W_{34}^{(1)}\\
W_{41}^{(1)}&W_{42}^{(1)}&W_{43}^{(1)}&W_{44}^{(1)}\\
\end{bmatrix},\\
&W_2=\begin{bmatrix}
{w}_{11}^{(2)}&w_{12}^{(2)}&w_{13}^{(2)}&W_{14}^{(2)}\\
w_{21}^{(2)}&w_{22}^{(2)}&w_{23}^{(2)}&W_{24}^{(2)}\\
w_{31}^{(2)}&w_{32}^{(2)}&w_{33}^{(2)}&W_{34}^{(2)}\\
W_{41}^{(2)}&W_{42}^{(2)}&W_{43}^{(2)}&W_{44}^{(2)}\\
\end{bmatrix},\quad W_3=\begin{bmatrix}
{w}_{11}^{(3)}&w_{12}^{(3)}&w_{13}^{(3)}&W_{14}^{(3)}\\
w_{21}^{(3)}&w_{22}^{(3)}&w_{23}^{(3)}&W_{24}^{(3)}\\
w_{31}^{(3)}&w_{32}^{(3)}&w_{33}^{(3)}&W_{34}^{(3)}\\
W_{41}^{(3)}&W_{42}^{(3)}&W_{43}^{(3)}&W_{44}^{(3)}\\
\end{bmatrix},\\
\end{aligned}$$
where $w_{ij}^{(s)}\in\mathbb{R},  W_{4j}^{(s)}\in\mathbb{R}^{(n-3)\times 1}, W_{i4}^{(s)}\in\mathbb{R}^{1\times (n-3)}$, $W_{44}^{(s)}\in\mathbb{R}^{(n-3)\times(n-3)}$, $s=0,1,2,3, i,j=1,2,3$. Let \begin{equation}\label{3-1}
\mathcal{R}(\mathbf{V})=:\begin{bmatrix}
V_0&-V_1&-V_2&-V_3\\
V_1&V_0&-V_3&V_2\\
V_2&V_3&V_0&-V_1\\
V_3&-V_2&V_1&V_0\\
\end{bmatrix}\in\mathbb{R}^{4n\times 4mk},\end{equation}
where \begin{equation}\label{3-2}\begin{aligned}
V_0&=[V_0^{(1)}, V_0^{(2)}, G_0^{(3)}, G_0^{(4)}, \cdots, G_0^{(k)}],\\
V_1&=[V_1^{(1)}, V_1^{(2)}, G_1^{(3)}, G_1^{(4)}, \cdots, G_1^{(k)}],\\
V_2&=[V_2^{(1)}, V_2^{(2)}, G_2^{(3)}, G_2^{(4)}, \cdots, G_2^{(k)}],\\
V_2&=[V_3^{(1)}, V_3^{(2)}, G_3^{(3)}, G_3^{(4)}, \cdots, G_3^{(k)}],\end{aligned}
\end{equation}
and $$\begin{aligned}
V_0^{(1)}&=\frac{1}{\sqrt{2}}\begin{bmatrix}
1&0&0\\
0&1&0\\
0&0&0^{(n-2)\times (m-2)}\\
\end{bmatrix},\quad V_1^{(1)}=V_2^{(1)}=V_3^{(1)}=0^{n\times m},\\
V_0^{(2)}&=\frac{1}{\sqrt{2}\gamma}\begin{bmatrix}
-\gamma^{(0)}&0&0\\
0&\gamma^{(0)}&0\\
0&0&0^{(n-2)\times (m-2)}\\
\end{bmatrix},\quad V_1^{(2)}=\frac{1}{\sqrt{2}\gamma}\begin{bmatrix}
-\gamma^{(1)}&0&0\\
0&\gamma^{(1)}&0\\
0&0&0^{(n-2)\times (m-2)}\\
\end{bmatrix},\\ V_2^{(2)}&=\frac{1}{\sqrt{2}\gamma}\begin{bmatrix}
-\gamma^{(2)}&0&0\\
0&\gamma^{(2)}&0\\
0&0&0^{(n-2)\times (m-2)}\\
\end{bmatrix},\quad  V_3^{(2)}=\frac{1}{\sqrt{2}\gamma}\begin{bmatrix}
-\gamma^{(3)}&0&0\\
0&\gamma^{(3)}&0\\
0&0&0^{(n-2)\times (m-2)}\\
\end{bmatrix},
\end{aligned}
$$
$\gamma=\sqrt{(\gamma^{(0)})^2+(\gamma^{(1)})^2+(\gamma^{(2)})^2+(\gamma^{(3)})^2}$ with $\gamma^{(s)}=w_{22}^{(s)}-w_{11}^{(s)} (s=0,1,2,3)$, and $G_s^{(l)}\in\mathbb{R}^{n\times m}$ are defined in \eqref{2-6}, $l=3,\cdots,k$.
It is easy to verify that $\mathcal{R}(\mathbf{V})$ with block column matrices is orthogonally JRS-symplectic. Then, 
by Eq.\eqref{2-5}, it follows that
$$
\mathcal{R}(\mathbf{V})^T\boxtimes(\mathcal{R}(\mathbf{W})\mathcal{R}(\mathbf{V})):=\begin{bmatrix}
\overline{H}_0&-\overline{H}_1&-\overline{H}_2&-\overline{H}_3\\
\overline{H}_1&\overline{H}_0&-\overline{H}_3&\overline{H}_2\\
\overline{H}_2&\overline{H}_3&\overline{H}_0&-\overline{H}_1\\
\overline{H}_3&-\overline{H}_2&\overline{H}_1&\overline{H}_0\\
\end{bmatrix}\in \mathbb{R}^{4k\times 4k},
$$
where $$\begin{aligned}
&\overline{H}_0=\begin{bmatrix}
\overline{w}_{11}^{(0)}&\overline{w}_{12}^{(0)}&\overline{w}_{13}^{(0)}&\overline{W}_{14}^{(0)}\\
\gamma^2&\overline{w}_{22}^{(0)}&\overline{w}_{23}^{(0)}&\overline{W}_{24}^{(0)}\\
0&\overline{w}_{32}^{(0)}&\overline{w}_{33}^{(0)}&W_{34}^{(0)}\\
0&0&\overline{W}_{43}^{(0)}&\overline{W}_{44}^{(0)}\\
\end{bmatrix},\quad \overline{H}_1=\begin{bmatrix}
\overline{w}_{11}^{(1)}&\overline{w}_{12}^{(1)}&\overline{w}_{13}^{(1)}&\overline{W}_{14}^{(1)}\\
0&\overline{w}_{22}^{(1)}&\overline{w}_{23}^{(1)}&\overline{W}_{24}^{(1)}\\
0&0&\overline{w}_{33}^{(1)}&\overline{W}_{34}^{(1)}\\
0&0&0&\overline{W}_{44}^{(1)}\\
\end{bmatrix},\\
&\overline{H}_2=\begin{bmatrix}
\overline{w}_{11}^{(2)}&\overline{w}_{12}^{(2)}&\overline{w}_{13}^{(2)}&\overline{W}_{14}^{(2)}\\
0&\overline{w}_{22}^{(2)}&\overline{w}_{23}^{(2)}&\overline{W}_{24}^{(2)}\\
0&0&\overline{w}_{33}^{(2)}&\overline{W}_{34}^{(2)}\\
0&0&0&\overline{W}_{44}^{(2)}\\
\end{bmatrix},\quad \overline{H}_3=\begin{bmatrix}
\overline{w}_{11}^{(3)}&\overline{w}_{12}^{(3)}&\overline{w}_{13}^{(3)}&\overline{W}_{14}^{(3)}\\
0&\overline{w}_{22}^{(3)}&\overline{w}_{23}^{(3)}&\overline{W}_{24}^{(3)}\\
0&0&\overline{w}_{33}^{(3)}&\overline{W}_{34}^{(3)}\\
0&0&0&\overline{W}_{44}^{(3)}\\
\end{bmatrix},\\
\end{aligned}$$
$\overline{W}_{43}^{(0)}=[\overline{w}_{43}^{(0)}, 0,\cdots, 0]^T\in\mathbb{R}^{(k-3)\times 1}$, and $\overline{W}_{44}^{(0)}\in \mathbb{R}^{(k-3)\times (k-3)}$ is of upper Hessenberg form, $\overline{W}_{44}^{(i)}\in \mathbb{R}^{(k-3)\times (k-3)}$ are upper triangular forms.
This shows that the assertion is surely true.
\end{proof}

Using the above results, we deduce the upper Hessenberg quaternion matrices by a quaternion matrix with unitary block column matrices.
\begin{theorem}\label{tho-2}
Let $\mathbf{W}=W_0+W_1\mathbf{i}+W_2\mathbf{j}+W_3\mathbf{k}$ be an $n\times n$ square quaternion matrix, and $k$ be a positive integer not greater than $n$. Then there exists a quaternion matrix $\mathbf{\mathcal{V}}_k=[\mathbf{V}_1,\mathbf{V}_2,\cdots,\mathbf{V}_k]\in\mathbb{Q}^{n\times nk}$  with unitary block column matrices, $\mathbf{V}_i\in\mathbb{Q}^{n\times n}$, $i=1,\cdots, k$, such that
\begin{equation}\label{3-3}
\mathbf{\mathcal{V}}_k^*\boxtimes(\mathbf{W}\mathbf{\mathcal{V}}_k)=\mathbf{H}_k
\end{equation}
is an upper Hessenberg quaternion matrix.
\end{theorem}
\begin{proof}
From Theorem \ref{tho-1}, for the real counterpart $\mathcal{R}(\mathbf{W})$ of $\mathbf{W}$,
there exists a JRS-symplectic matrix $\mathcal{R}(\mathbf{V})$ whose block column matrices are mutually orthogonal, such that
\begin{equation}\label{3-4}
\mathcal{R}(\mathbf{V})^T\boxtimes (\mathcal{R}(\mathbf{W})\mathcal{R}(\mathbf{V}))=\overline{H}\in \mathbb{R}^{4k\times 4k}
\end{equation}
is of upper JRS-Hessenberg form. Denote $\mathbf{H}_k=\mathcal{R}^{-1}(\overline{H})$, $\mathbf{\mathcal{V}}_k=\mathcal{R}^{-1}(\mathcal{R}(\mathbf{V}))$, i.e., $\mathbf{V}_p=V_0^{(p)}+V_1^{(p)}\mathbf{i}+V_2^{(p)}\mathbf{j}+V_3^{(p)}\mathbf{k}, \mathbf{V}_l=G_0^{(l)}+G_1^{(l)}\mathbf{i}+G_2^{(l)}\mathbf{j}+G_3^{(l)}\mathbf{k},$ $p=1,2, l=3,\cdots,k$. Applying the inverse linear homeomorphic mapping $\mathcal{R}^{-1}$, we can see that Eq.\eqref{3-3} holds.
\end{proof}

As stated by \cite{Jia1}, all real counterparts of $\mathbf{A}$  are permutationally equivalent. This shows that Theorem \ref{tho-2} holds for different real counterparts of $\mathbf{A}$. Moreover, similar to \cite{Jia4}, we call that the decomposition in \eqref{3-4} is \emph{structure-preserving}, since $\overline{H}$ here inherits the JRS-symmetry of $\mathcal{R}(\mathbf{W})$. From a numerical point of view, we only generate and store the four parts of quaternion matrices, rather than their real counterparts, which saves the computational operations and storage.

 From the above discussions, we now derive the global quaternion Arnoldi method.
Let $\mathbf{A}\in \mathbb{Q}^{m\times m}$ and $\mathbf{V}\in \mathbb{Q}^{n\times m}$ be the quaternion matrices, and let $k$ be a positive integer. The quaternion matrix Krylov subspace is of the form
\begin{equation}\label{3-5}
\mathcal{K}_k(\mathbf{A},\mathbf{V}):=\mbox{span}\{\mathbf{V},\mathbf{A}\mathbf{V},\cdots,\mathbf{A}^{k-1}\mathbf{V}\}.
\end{equation} 
Each element is only written as a right-hand side linear combination of $\mathbf{V},\mathbf{A}\mathbf{V},\cdots,$ $\mathbf{A}^{k-1}\mathbf{V}$, rather than a polynomial of degree $k-1$ formed by them. As shown in \cite{Jia4}, we recall that the grade of $\mathbf{V}$ with $\mathbf{A}$ is the smallest positive integer $\bm{\mu}$, i.e., the dimension of $\mathcal{K}_k(\mathbf{A},\mathbf{V})$, such that $\mathbf{V},\mathbf{A}\mathbf{V},\cdots,\mathbf{A}^{\bm{\mu}-1}\mathbf{V}$ are linearly independent, but $\mathbf{V},\mathbf{A}\mathbf{V},\cdots,\mathbf{A}^{\bm{\mu}}\mathbf{V}$ are linearly dependent. 

The global quaternion Arnoldi method forms an orthonormal basis, i.e.,
\begin{equation}\label{3-6}
\langle{\mathbf{V}}_{i},{\mathbf{V}}_{j}\rangle:=\left\{
 \begin{aligned}
 	&1,\quad\quad \mbox{if}\quad i=j,\\
&0,\quad\quad \mbox{if}\quad i\neq j,\\
 		\end{aligned}
 \right.
\end{equation}
of the quaternion Krylov subspace $\mathcal{K}_k(\mathbf{A},\mathbf{V})$. The method is expounded in detail as follows.
\bigskip
\begin{breakablealgorithm}
	\caption{ Global Quaternion Arnoldi Method}\label{algorithm-1}
	\begin{algorithmic}[1]
		 \STATE Choose an initial guess $\mathbf{V}_1$ such that $\|\mathbf{V}_1\|=1$.\\
		\STATE 	For $j=1,2,\cdots,k$, Do:\\
	\STATE Compute $\mathbf{W}_j=\mathbf{A}\mathbf{V}_j$\\
		\STATE  Compute $\mathbf{h}_{ij}=\langle\mathbf{W}_j,\mathbf{V}_i\rangle=\mbox{tr}( \mathbf{V}_i^*\mathbf{W}_j)$ for $i=1,\cdots,j$\\
		\STATE  Compute $\mathbf{W}_j=\mathbf{W}_j-\sum_{i=1}^j\mathbf{V}_i\mathbf{h}_{ij}$\\
		\STATE $\mathbf{h}_{j+1,j}=\|\mathbf{W}_j\|$. If $\mathbf{h}_{j+1,j}=0$ Stop\\
		\STATE $\mathbf{V}_{j+1}=\mathbf{W}_{j}/\mathbf{h}_{j+1,j}$\\
		\STATE End Do\\
	\end{algorithmic}
\end{breakablealgorithm}

Note that the key point of Algorithm \ref{algorithm-1} is to compute $\mathbf{A}\mathbf{V}_j$ and $\langle\mathbf{W}_j,\mathbf{V}_i\rangle$. Using Eq.\eqref{1-1}, the quaternion matrix-matrix multiplication and the inner product are computed by
\begin{equation}\label{3-7}
\begin{aligned}
\mathbf{W}_j&=\mathcal{R}^{-1}(\mathcal{R}(\mathbf{A})\mathcal{R}(\mathbf{V}_j)),\\
\langle\mathbf{W}_j,\mathbf{V}_i\rangle&=\mbox{tr}(\mathcal{R}^{-1}(\mathcal{R}(\mathbf{V}_i^*)\mathcal{R}(\mathbf{W}_j))).
\end{aligned}
\end{equation}
Indeed, we can achieve the above relations once 
the four parts consisting of the first block column of their real counterparts are obtained. 
Let $V_t$ be the $t$-th block column of the real counterpart $\mathcal{R}(\mathbf{V}_i)$ of $\mathbf{V}_i$, $t=1,2,3,4$, and $\mathbf{h}_{ij}=h_0+h_1\mathbf{i}+h_2\mathbf{j}+h_3\mathbf{k}$. Then $\mathbf{W}_j$ and $\mathbf{h}_{ij}$ are computed by
\begin{equation}\label{3-8}
\begin{aligned}
\mathcal{R}(\mathbf{W}_j)_{c}&=\mathcal{R}(\mathbf{A})\mathcal{R}(\mathbf{V}_j)_{c},\\
h_{t-1}&=\mbox{tr}(V_t^T\mathcal{R}(\mathbf{W}_j)_{c}),
\end{aligned}
\end{equation}
respectively. This saves three-quarters of the theoretical costs, since the above matrix-matrix multiplication and inner product cost only $4mn(8n-1)$ and $4m(8n-1)$ real operations, respectively. Additionally, we have the following statements for Algorithm \ref{algorithm-1}.

\begin{proposition}\label{pro-1}
Assume that Algorithm \ref{algorithm-1} does not stop before the $k$-th step. Then the block system $\{\mathbf{V}_1,\mathbf{V}_2,\cdots,\mathbf{V}_k\}$ forms an orthonormal basis of the quaternion matrix Krylov subspace $\mathcal{K}_k(\mathbf{A},\mathbf{V}_1)$. 
\end{proposition}
\begin{proof}
The proof is similar to those presented by \cite{Saad1} and is omitted.
\end{proof}
\begin{proposition}\label{pro-2}
Let $\mathbf{\mathcal{V}}_k=[\mathbf{V}_1,\mathbf{V}_2,\cdots,\mathbf{V}_k]$ be a block matrix, where $\{\mathbf{V}_1,\mathbf{V}_2,\cdots,\mathbf{V}_k\}$ gengerated by Algorithm \ref{algorithm-1} forms an orthonormal basis of $\mathcal{K}_k(\mathbf{A},\mathbf{V}_1),$  and let $\bm{\alpha}=(\bm{\alpha}_1,\bm{\alpha}_2,\cdots,\bm{\alpha}_k)^T\in\mathbb{Q}^{k}$. Then we have
\begin{equation}\label{3-9}
\|\mathbf{\mathcal{V}}_k*\bm{\alpha}\|=\|\bm{\alpha}\|.
\end{equation}
\end{proposition}
\begin{proof}
The proof is similar to that of \cite{Jbilou} and is omitted.
\end{proof}
\begin{proposition}\label{pro-3}
Let $\mathbf{\mathcal{V}}_k$ be defined as before, and $\overline{\mathbf{H}}_k$ be the $(k+1)\times k$ upper Hessenberg quaternion matrix whose nonzero entries $\mathbf{h}_{ij}$ are computed by Algorithm \ref{algorithm-1}, and let $\mathbf{H}_k$ be the $k\times k$
quaternion matrix obtained by deleting its last row. Then the following relations hold:
\begin{equation}\label{3-10}
\begin{aligned}
\mathbf{A}\mathbf{\mathcal{V}}_k&=\mathbf{\mathcal{V}}_k*\mathbf{H}_k+[0^{n\times m},\cdots,0^{n\times m},\mathbf{V}_{k+1}]\mathbf{h}_{k+1,k}\\
&=\mathbf{\mathcal{V}}_{k+1}*\overline{\mathbf{H}}_k,\\
\mathbf{\mathcal{V}}_k^*\boxtimes(\mathbf{A}\mathbf{\mathcal{V}}_k)&=\mathbf{H}_k,
\end{aligned}
\end{equation}
where $\mathbf{\mathcal{V}}_{k+1}=[\mathbf{\mathcal{V}}_{k},\mathbf{V}_{k+1}]$.
\end{proposition}
\begin{proof}
The first term follows by lines 4-7 of Algorithm \ref{algorithm-1}. For the second term, it follows that 
$$\begin{aligned}
\mathbf{\mathcal{V}}_k^*\boxtimes(\mathbf{A}\mathbf{\mathcal{V}}_k)&=\mathbf{\mathcal{V}}_k^*\boxtimes(\mathbf{\mathcal{V}}_k*\mathbf{H}_k+[0^{n\times m},\cdots,0^{n\times m},\mathbf{V}_{k+1}]\mathbf{h}_{k+1,k})\\
&=\mathbf{\mathcal{V}}_k^*\boxtimes[\sum_{i=1}^k\mathbf{V}_i\mathbf{h}_{i1}, \sum_{i=1}^k\mathbf{V}_i\mathbf{h}_{i2},\cdots,\sum_{i=1}^k\mathbf{V}_i\mathbf{h}_{ik}]\\
&=\mathbf{H}_k.
\end{aligned}$$
\end{proof}

\begin{proposition}\label{pro-4}
Algorithm \ref{algorithm-1}  breaks down at step $k$ if and only if the grade of $\mathbf{V}_1$ with $\mathbf{A}$ is $k$. In this case, the the subspace $\mathcal{K}_k(\mathbf{A},\mathbf{V}_1)$ is invariant under $\mathbf{A}$. 
\end{proposition}
\begin{proof}
The proof is similar to those given by \cite{Jia4} and is omitted.
\end{proof}

It is remarkable that Algorithm \ref{algorithm-1} based on the Gram-Schmidt process is not reliable numerically. To overcome this issue, the algorithm by the modified Gram-Schmidt alternative is as follows.
\bigskip
\begin{breakablealgorithm}
	\caption{Modified Global Quaternion Arnoldi Method}\label{algorithm-2}
	\begin{algorithmic}[1]
		 \STATE Choose an initial guess $\mathbf{V}_1$ such that $\|\mathbf{V}_1\|=1$.\\
		\STATE 	For $j=1,2,\cdots,k$, Do:\\
	\STATE Compute $\mathbf{W}_j=\mathbf{A}\mathbf{V}_j$\\
\STATE For $i=1,2,\cdots,j$, Do:
		\STATE  Compute $\mathbf{h}_{ij}=\langle\mathbf{W}_j,\mathbf{V}_i\rangle=\mbox{tr}( \mathbf{V}_i^*\mathbf{W}_j)$\\
		\STATE  Compute $\mathbf{W}_j=\mathbf{W}_j-\sum_{i=1}^j\mathbf{V}_i\mathbf{h}_{ij}$\\
  	\STATE End Do\\
		\STATE $\mathbf{h}_{j+1,j}=\|\mathbf{W}_j\|$. If $\mathbf{h}_{j+1,j}=0$ Stop\\
		\STATE $\mathbf{V}_{j+1}=\mathbf{W}_{j}/\mathbf{h}_{j+1,j}$\\
		\STATE End Do\\
	\end{algorithmic}
\end{breakablealgorithm}
\bigskip
In exact arithmetic, the above algorithm performs
much more reliable than Algorithm \ref{algorithm-1} with the presence of round-off, even though they are numerically equivalent.
 Compared with Algorithm \ref{algorithm-1}, this fact shows that Algorithm \ref{algorithm-2} is more popular in practice. 

\section{Global quaternion FOM and GMRES methods}\label{section-4}

In this section, we develop the global quaternion full orthogonalization (Gl-QFOM) and global quaternion generalized minimum residual (Gl-QGMRES) methods, which are based on Algorithm \ref{algorithm-2} and orthogonal and oblique projections, for solving the block linear systems \eqref{1-1}. Then we apply the proposed methods for solving the Sylvester quaternion matrix equation.

Let $\mathbf{X}_0$ be an initial guess to Eq.\eqref{1-1} and $\mathbf{R}_0=\mathbf{B}-\mathbf{A}\mathbf{X}_0$ be its corresponding residual. We now consider the Gl-QFOM method based on the orthogonal projection, which takes $\mathcal{L}_k=\mathcal{K}_k$ with $$\mathcal{K}_k=\mathcal{K}_k(\mathbf{A},\mathbf{R}_0):=\mbox{span}\{\mathbf{R}_0,\mathbf{A}\mathbf{R}_0,\cdots,\mathbf{A}^{k-1}\mathbf{R}_0\}.$$ Let $\mathbf{V}_1=\frac{\mathbf{R}_0}{\|\mathbf{R}_0\|}$ and $\beta=\|\mathbf{R}_0\|.$ It follows from Algorithm \ref{algorithm-2} that 
$$
\mathbf{\mathcal{V}}_k^*\boxtimes\mathbf{R}_0=\mathbf{\mathcal{V}}_k^*\boxtimes(\beta\mathbf{V}_1)=\beta\mathbf{e}_1,
$$
where $\mathbf{e}_1=[1,0,\cdots,0]^T\in\mathbb{R}^{k}$. This method seeks an approximate solution $\mathbf{X}_k$ from the
affine subspace $\mathbf{X}_0+\mathcal{K}_k$ of dimension $k$ by imposing the condition that the new residual $\mathbf{X}_k$ is orthogonal to $\mathcal{L}_k$. In other words, the approximate solution, defined as 
\begin{equation}\label{4-0}
\mathbf{X}_k=\mathbf{X}_0+\mathbf{\mathcal{V}}_k*\mathbf{y}_k,\quad \mathbf{\mathcal{V}}_k*\mathbf{y}_k\in \mathcal{K}_k,
\end{equation}
yields
$$
\langle\mathbf{R}_0-\mathbf{A}(\mathbf{\mathcal{V}}_k*\mathbf{y}_k),\Omega\rangle=0,\quad  \forall\ \Omega\in \mathcal{K}_k,
$$
where $\mathbf{y}_k$ is a quaternion vector of order $n$.
From Eqs.\eqref{2-1}, \eqref{2-2} and \eqref{2-3}, it follows that
\begin{equation}\label{4-1}
0=\mathbf{\mathcal{V}}_k^*\boxtimes(\mathbf{R}_0-\mathbf{A}(\mathbf{\mathcal{V}}_k*\mathbf{y}_k))=\mathbf{\mathcal{V}}_k^*\boxtimes\mathbf{R}_0-(\mathbf{\mathcal{V}}_k^*\boxtimes(\mathbf{A}\mathbf{\mathcal{V}}_k))\mathbf{y}_k=\beta\mathbf{e}_1-\mathbf{H}_k\mathbf{y}_k.
\end{equation}
As a result, the  approximate solution is given by
\begin{equation}\label{4-2}
\begin{aligned}
\mathbf{X}_k&=\mathbf{X}_0+\mathbf{\mathcal{V}}_k*\mathbf{y}_k,\\
\mathbf{y}_k&=\mathbf{H}_k^{-1}(\beta\mathbf{e}_1).
\end{aligned}
\end{equation}
A vital point of the above relations is choosing a suitable parameter $k$, i.e., the dimension of $\mathcal{K}_k$. In practice, the selected parameter $k$ is desirable if the residual norm to $\mathbf{X}_k$ is available inexpensively. With this information, the method will be terminated at the appropriate step. The following theorem presents the results to this point.
\begin{theorem}\label{tho-4-1}
The residual vector of the approximate solution $\mathbf{X}_k$ obtained by Eq.\eqref{4-2} yields
\begin{equation}\label{4-3}
\mathbf{B}-\mathbf{A}\mathbf{X}_k=-\mathbf{h}_{k+1,k}\mathbf{V}_{k+1}\mathbf{y}_k(k),
\end{equation}
in which $\mathbf{y}_k(k)$ represents the last element of the quaternion vector $\mathbf{y}_k$, and, therefore,
\begin{equation}\label{4-4}
\|\mathbf{B}-\mathbf{A}\mathbf{X}_k\|=\mathbf{h}_{k+1,k}|\mathbf{y}_k(k)|.
\end{equation}
\end{theorem}
\begin{proof}
   From Eqs. \eqref{2-3}, \eqref{3-10} and \eqref{4-2}, it follows that
   $$
   \begin{aligned}
   \mathbf{B}-\mathbf{A}\mathbf{X}_k&=\mathbf{R}_0-(\mathbf{A}\mathbf{\mathcal{V}}_k)*\mathbf{y}_k\\
   &=\mathbf{R}_0-(\mathbf{\mathcal{V}}_k*\mathbf{H}_k+[0^{n\times m},\cdots,0^{n\times m},\mathbf{V}_{k+1}]\mathbf{h}_{k+1,k})*\mathbf{y}_k\\
   &=\mathbf{R}_0-\mathbf{\mathcal{V}}_k*(\mathbf{H}_k\mathbf{y}_k)-\mathbf{h}_{k+1,k}[0^{n\times m},\cdots,0^{n\times m},\mathbf{V}_{k+1}]*\mathbf{y}_k,\\
   &=-\mathbf{h}_{k+1,k}\mathbf{V}_{k+1}\mathbf{y}_k(k).
   \end{aligned}$$
   The second term follows immediately by taking the norm on both sides of Eq.\eqref{4-3}.
\end{proof}

Based on this approach and Algorithm \ref{algorithm-2}, a method, termed the global quaternion full orthogonalization method (Gl-QFOM), is described as follows.
\bigskip
\begin{breakablealgorithm}
	\caption{ Global Quaternion Full Orthogonalization Method}\label{algorithm-3}
	\begin{algorithmic}[1]
		 \STATE Choose an initial guess $\mathbf{X}_0$, and compute $\mathbf{R}_0=\mathbf{B}-\mathbf{A}\mathbf{X}_0$, $\beta=\|\mathbf{R}_0\|$ and $\mathbf{V}_1=\mathbf{R}_0/\beta$.\\
		\STATE 	For $j=1,2,\cdots$, Do:\\
	\STATE Compute $\mathbf{W}_j=\mathbf{A}\mathbf{V}_j$\\
\STATE For $i=1,2,\cdots,j$, Do:
		\STATE  Compute $\mathbf{h}_{ij}=\langle\mathbf{W}_j,\mathbf{V}_i\rangle=\mbox{tr}( \mathbf{V}_i^*\mathbf{W}_j)$\\
		\STATE  Compute $\mathbf{W}_j=\mathbf{W}_j-\sum_{i=1}^j\mathbf{V}_i\mathbf{h}_{ij}$\\
  	\STATE End Do\\
		\STATE $\mathbf{h}_{j+1,j}=\|\mathbf{W}_j\|$. If $\mathbf{h}_{j+1,j}=0$ Stop\\
		\STATE  Solve the quaternion linear systems $\mathbf{H}_j\mathbf{y}_j=\beta\mathbf{e}_1$, i.e., $\mathbf{y}_j=\beta\mathbf{H}_j^{-1}\mathbf{e}_1$
  \STATE If  $\mathbf{h}_{j+1,j}|\mathbf{y}_j(j)|/\beta<\epsilon$, then stop
  \STATE $\mathbf{V}_{j+1}=\mathbf{W}_{j}/\mathbf{h}_{j+1,j}$\\
		\STATE End Do\\
  \STATE Compute the approximate solution $\mathbf{X}_j=\mathbf{X}_0+\mathbf{\mathcal{V}}_j*\mathbf{y}_j.$
	\end{algorithmic}
\end{breakablealgorithm}
\bigskip

Observe that the main cost of Algorithm \ref{algorithm-3} is to solve 
\begin{equation}\label{4-5}
\mathbf{H}_j\mathbf{y}_j=\beta\mathbf{e}_1
\end{equation}
with the iterations increasing. As shown in \cite{Li0}, by the
generalized quaternion Givens rotations, we can transform the upper Hessenberg quaternion matrix $\mathbf{H}_j$ to an upper triangular quaternion matrix, such that the computational operations for obtaining $\mathbf{y}_j$ are greatly reduced.
Suppose that $\mathbf{h}_{i+1,i}\neq 0$ and  $\mathbf{r}_{ii}=\|\left[\begin{array}{ccccc}
	\mathbf{h}_{ii}&
	\mathbf{h}_{i+1,i}\\
\end{array}\right]^T\|\neq 0$ for $i=1,\cdots,j-1.$ Then the $i$-th generalized quaternion Givens rotation \cite{Jia3} is defined as 
\begin{equation}\label{4-6}
\mathbf{G}_i=\left[\begin{array}{ccccc}
	\mathbf{I}^{(i-1)\times(i-1)}&&&\\	&\mathbf{g}_{11}^{(i)}&\mathbf{g}_{12}^{(i)}&\\
	&\mathbf{g}_{21}^{(i)}&\mathbf{g}_{22}^{(i)}&\\
	&&&\mathbf{I}^{(j-i-1)\times (j-i-1)}\\
\end{array}\right]\in\mathbb{Q}^{j\times j},
\end{equation}
in which
$$
\mathbf{g}_{11}^{(i)}=\frac{\mathbf{h}_{ii}}{\mathbf{r}_{ii}},\ \mathbf{g}_{21}^{(i)}=\frac{\mathbf{h}_{i+1,i}}{\mathbf{r}_{ii}},
$$
$$\left\{
 \begin{aligned}
\mathbf{g}_{12}^{(i)}&=|\mathbf{g}_{21}^{(i)}|,\  \mathbf{g}_{22}^{(i)}=-\frac{\mathbf{g}_{21}^{(i)}}{|\mathbf{g}_{21}^{(i)}|}(\mathbf{g}_{11}^{(i)})^*,\ \ \mbox{if}\  |\mathbf{h}_{ii}|\leq |\mathbf{h}_{i+1,i}|,\\
\mathbf{g}_{22}^{(i)}&=|\mathbf{g}_{11}^{(i)}|,\  \mathbf{g}_{12}^{(i)}=-\frac{\mathbf{g}_{11}^{(i)}}{|\mathbf{g}_{11}^{(i)}|}(\mathbf{g}_{21}^{(i)})^*,\ \ \mbox{if}\ |\mathbf{h}_{ii}|>|\mathbf{h}_{i+1,i}|.\\
 		\end{aligned}
 \right.$$
Consequently, we can see that Eq.\eqref{4-5} reduces to the following linear systems
\begin{equation}\label{4-7}
\overline{\mathbf{R}}_j\mathbf{y}_j=\mathbf{Q}_{j-1}^*(\beta\mathbf{e}_1),
\end{equation} 
in which $\mathbf{Q}_{j-1}=\mathbf{G}_1\mathbf{G}_2\cdots \mathbf{G}_{j-1}$, and $\overline{\mathbf{R}}_j=\mathbf{Q}_{j-1}^*\mathbf{H}_j\in \mathbb{Q}^{j\times j}$ is an upper triangular form whose entries $\mathbf{r}_{ss}$ and $\mathbf{r}_{jj}$ are positive real numbers and quaternion scalar, $s=1,\cdots,j-1$, respectively. This gives rise to a practical Gl-QFOM method for solving the original linear systems. 
\bigskip
\begin{breakablealgorithm}
	\caption{ Practical Global Quaternion Full Orthogonalization Method}\label{algorithm-4}
	\begin{algorithmic}[1]
		 \STATE Choose an initial guess $\mathbf{X}_0$, and compute $\mathbf{R}_0=\mathbf{B}-\mathbf{A}\mathbf{X}_0$, $\beta=\|\mathbf{R}_0\|$ and $\mathbf{V}_1=\mathbf{R}_0/\beta$, $\mathbf{u}=\beta\mathbf{e}_1$.\\
		\STATE 	For $j=1,2,\cdots$, Do:\\
	\STATE Compute $\mathbf{W}_j=\mathbf{A}\mathbf{V}_j$\\
\STATE For $i=1,2,\cdots,j$, Do:
		\STATE  Compute $\mathbf{h}_{ij}=\langle\mathbf{W}_j,\mathbf{V}_i\rangle=\mbox{tr}( \mathbf{V}_i^*\mathbf{W}_j)$\\
		\STATE  Compute $\mathbf{W}_j=\mathbf{W}_j-\sum_{i=1}^j\mathbf{V}_i\mathbf{h}_{ij}$\\
  	\STATE End Do\\
\STATE For $i=1,\cdots,j-1$, Do:
$$\left[\begin{array}{ccccc}
	\mathbf{h}_{i,j}\\
	\mathbf{h}_{i+1,j}\\
\end{array}\right]=\left[\begin{array}{ccccc}
\mathbf{g}_{11}^{(i)}&\mathbf{g}_{12}^{(i)}\\
\mathbf{g}_{21}^{(i)}&\mathbf{g}_{22}^{(i)}\\
\end{array}\right]^*\left[\begin{array}{ccccc}
	\mathbf{h}_{i,j}\\
	\mathbf{h}_{i+1,j}\\
\end{array}\right]$$
\STATE End Do
\STATE { $\mathbf{y}_j(j)=\mathbf{h}_{jj}^{-1}\mathbf{u}(j)$, i.e., $|\mathbf{y}_j(j)|=|\mathbf{u}(j)|/|\mathbf{h}_{jj}|$}
\STATE $\mathbf{h}_{j+1,j}=\|\mathbf{W}_j\|$. If $\mathbf{h}_{j+1,j}=0$ Stop\\
 \STATE If  $\mathbf{h}_{j+1,j}|\mathbf{y}_j(j)|/\beta<\epsilon$, then stop
  \STATE $\mathbf{V}_{j+1}=\mathbf{W}_{j}/\mathbf{h}_{j+1,j}$\\
\STATE If $|\mathbf{h}_{jj}|\leq |\mathbf{h}_{j+1,j}|$, then
\STATE $\mathbf{g}_{12}^{(j)}=|\mathbf{g}_{21}^{(j)}|,\  \mathbf{g}_{22}^{(j)}=-\frac{\mathbf{g}_{21}^{(j)}}{|\mathbf{g}_{21}^{(i)}|}(\mathbf{g}_{11}^{(j)})^*$, where $\mathbf{g}_{11}^{(j)}=\frac{\mathbf{h}_{jj}}{\mathbf{r}_{jj}},\ \mathbf{g}_{21}^{(j)}=\frac{\mathbf{h}_{j+1,j}}{\mathbf{r}_{jj}}$
\STATE Else
\STATE$\mathbf{g}_{22}^{(j)}=|\mathbf{g}_{11}^{(j)}|,\  \mathbf{g}_{12}^{(j)}=-\frac{\mathbf{g}_{11}^{(j)}}{|\mathbf{g}_{11}^{(j)}|}(\mathbf{g}_{21}^{(j)})^*$, where $\mathbf{g}_{11}^{(j)}=\frac{\mathbf{h}_{jj}}{\mathbf{r}_{jj}},\ \mathbf{g}_{21}^{(j)}=\frac{\mathbf{h}_{j+1,j}}{\mathbf{r}_{jj}}$
\STATE
\STATE $\mathbf{h}_{jj}=\|\left[\begin{array}{ccccc}
	\mathbf{h}_{jj}&
	\mathbf{h}_{j+1,j}\\
\end{array}\right]^T\|_2$, $\mathbf{h}_{j+1,j}=0$
\STATE $$\left[\begin{array}{ccccc}
	\mathbf{u}(j)\\
	\mathbf{u}(j+1)\\
\end{array}\right]=\left[\begin{array}{ccccc}
	\mathbf{g}_{11}^{(j)}&\mathbf{g}_{12}^{(j)}\\
	\mathbf{g}_{21}^{(j)}&\mathbf{g}_{22}^{(j)}\\
\end{array}\right]^*\left[\begin{array}{ccccc}
	\mathbf{u}(j)\\
	0\\
\end{array}\right]$$
\STATE End Do
\STATE  Solve the quaternion linear systems $\overline{\mathbf{R}}_j\mathbf{y}_j=\mathbf{u}$
  \STATE Compute the approximate solution $\mathbf{X}_j=\mathbf{X}_0+\mathbf{\mathcal{V}}_j*\mathbf{y}_j.$
	\end{algorithmic}
\end{breakablealgorithm}
\bigskip

From Theorem \ref{tho-4-1}, we can see that if the Gl-QAP of Algorithm \ref{algorithm-4} breaks down at $j$, i.e., $\mathbf{h}_{j+1,j}=0$, then its residual norm $\|\mathbf{B}-\mathbf{A}\mathbf{X}_j\|=0$, i.e., the approximate solution $\mathbf{X}_j$ is exact. The Gl-QAP will not loop infinitely since the dimension of $\mathcal{K}_j(\mathbf{A},\mathbf{R}_0)$ is at most $4mn$. This gives rise to the following theorem.
\begin{theorem}\label{tho-4-0}
Suppose that $\mathbf{X}^*$ is an exact solution of the original linear systems Eq.\eqref{1-1}. Then an iterative sequence generated by Algorithm \ref{algorithm-4} converges to $\mathbf{X}^*$ within at most $4mn$ iterations in the presence of round-off errors.
\end{theorem}
\begin{proof}
The proof is similar to Theorem 4 given by \cite{Li0} and is omitted.
\end{proof}

In what follows, by taking the $\mathcal{K}_j=\mathcal{K}_j(\mathbf{A},\mathbf{R}_0)$ and $\mathcal{L}_k=\mathbf{A}\mathcal{K}_j$, we derive the global quaternion generalized minimum residual (Gl-QGMRES) method for solving Eq.\eqref{1-1}. Such a method based on the oblique technique minimizes the residual norm at each step. As noted earlier, the approximate solution $\mathbf{X}_j$ in $\mathbf{X}_0+\mathcal{K}_j$ can be written as Eq.\eqref{4-0}. The Gl-QGMRES approximation is to search the unique quaternion matrix $\mathbf{X}_j$ yielding
$$\argmin_{\mathbf{X}_j\in\mathbf{X}_0+\mathcal{K}_j}\|\mathbf{B}-\mathbf{A}\mathbf{X}_j\|.$$
According to Eqs. \eqref{3-9} and \eqref{3-10}, it follows that
$$\begin{aligned}
\|\mathbf{B}-\mathbf{A}\mathbf{X}_j\|&=\|\mathbf{B}-\mathbf{A}(\mathbf{X}_0+\mathbf{\mathcal{V}}_j*\mathbf{y}_j)\|\\
&=\|\beta\mathbf{V}_1-\mathbf{\mathcal{V}}_{j+1}*(\overline{\mathbf{H}}_j\mathbf{y}_j)\|\\
&=\|\mathbf{\mathcal{V}}_{j+1}*(\beta\mathbf{\overline{e}}_1-\overline{\mathbf{H}}_j\mathbf{y}_j)\|\\
&=\|\beta\mathbf{\overline{e}}_1-\overline{\mathbf{H}}_j\mathbf{y}_j\|.
\end{aligned}$$
Therefore, the above minimization problem can be rewritten as
\begin{equation}\label{4-8}
\begin{aligned}
\mathbf{X}_j&=\mathbf{X}_0+\mathbf{\mathcal{V}}_j*\mathbf{y}_j, \mbox{where}\\
\mathbf{y}_j&=\argmin_{y}\|\beta\mathbf{\overline{e}}_1-\overline{\mathbf{H}}_j\mathbf{y}_j\|
\end{aligned}
\end{equation}
with $\mathbf{\overline{e}}_1=[1,0,\cdots,0]\in\mathbb{R}^{j+1}$.
 Since we only need to obtain the solution of a
$(j+1)\times j$ least-squares problem where $j$ is typically small, the minimizer $\mathbf{y}_j$ is inexpensive to compute. Let $\mathbf{G}_i\in\mathbb{Q}^{(j+1)\times (j+1)}$ be defined as Eq.\eqref{4-7}, $i=1,\cdots,j$ and $\mathbf{Q}_j=\mathbf{G}_1\mathbf{G}_2\cdots\mathbf{G}_j$. Then, from \cite{Jia3}, the QR factorization of $\overline{\mathbf{H}}_j$ is given by
$$
\overline{\mathbf{H}}_j=\mathbf{Q}_j^*\widetilde{\mathbf{R}}_{j+1},
$$
where $\widetilde{\mathbf{R}}_{j+1}=\left[\begin{array}{ccccc}
\widetilde{\mathbf{R}}_{j}\\
0\\
\end{array}\right]\in\mathbb{Q}^{(j+1)\times j},$
and $\widetilde{\mathbf{R}}_{j}$ obtained from $\widetilde{\mathbf{R}}_{j+1}$ by deleting its last row, is an $j\times j$ upper triangular quaternion matrix whose diagonal entries are positive real numbers. The above formulation results in
\begin{equation}\label{4-10}
\begin{aligned}
\|\beta\mathbf{\overline{e}}_1-\overline{\mathbf{H}}_j\mathbf{y}_j\|&=\|\beta\mathbf{\overline{e}}_1-\mathbf{Q}_j^*\widetilde{\mathbf{R}}_{j+1}\mathbf{y}_j\|\\
&=\|\beta\mathbf{q}_1-\left[\begin{array}{ccccc}
\widetilde{\mathbf{R}}_{j}\\
0\\
\end{array}\right]\mathbf{y}_j\|,
\end{aligned}
\end{equation}
where $\mathbf{q}_1$ of order $j+1$ denotes the first column of $\mathbf{Q}_j^*$. This shows that the desired $\mathbf{y}_j$ is obtained by
$$\widetilde{\mathbf{R}}_{j}\mathbf{y}_j=\beta\mathbf{q}_1(1:j),$$
in which $\mathbf{q}_1(1:j)$ denotes the first $j$ elements of $\mathbf{q}_1$. From the above discussions, the following theorem gives the norm of the residual vector at each iteration.

\begin{theorem}\label{tho-4-2}
Let $\mathbf{X}_j$ be an approximate solution generated by the Gl-QGMRES method at step $j$. Then we have
\begin{equation}\label{4-11}
\|\mathbf{B}-\mathbf{A}\mathbf{X}_j\|=\beta|\mathbf{q}_1(j+1)|,
\end{equation}
where $\mathbf{q}_1(j+1)$ is the last element of $\mathbf{q}_1$.
\end{theorem}
\begin{proof}
    The proof of this is straightforward.
\end{proof}

Now we develop the practical global quaternion generalized minimal residual method as follows.

\bigskip
\begin{breakablealgorithm}
	\caption{ Practical Global Quaternion Generalized Minimal Residual Method }\label{algorithm-5}
	\begin{algorithmic}[1]
		 \STATE Choose an initial guess $\mathbf{X}_0$, and compute $\mathbf{R}_0=\mathbf{B}-\mathbf{A}\mathbf{X}_0$, $\beta=\|\mathbf{R}_0\|$ and $\mathbf{V}_1=\mathbf{R}_0/\beta$, $\mathbf{u}=\beta\mathbf{\overline{e}}_1$.\\
		\STATE 	For $j=1,2,\cdots$, Do:\\
		\STATE The quaternion Arnoldi method using lines 3-7 of Algorithm \ref{algorithm-4}.
\STATE For $i=1,\cdots,j-1$, Do:
$$\left[\begin{array}{ccccc}
	\mathbf{h}_{i,j}\\
	\mathbf{h}_{i+1,j}\\
\end{array}\right]=\left[\begin{array}{ccccc}
\mathbf{g}_{11}^{(i)}&\mathbf{g}_{12}^{(i)}\\
\mathbf{g}_{21}^{(i)}&\mathbf{g}_{22}^{(i)}\\
\end{array}\right]^*\left[\begin{array}{ccccc}
	\mathbf{h}_{i,j}\\
	\mathbf{h}_{i+1,j}\\
\end{array}\right]$$
\STATE End Do
\STATE $\mathbf{h}_{j+1,j}=\|\mathbf{W}_j\|$. If $\mathbf{h}_{j+1,j}=0$ Stop\\
  \STATE $\mathbf{V}_{j+1}=\mathbf{W}_{j}/\mathbf{h}_{j+1,j}$\\
\STATE If $|\mathbf{h}_{jj}|\leq |\mathbf{h}_{j+1,j}|$, then
\STATE $\mathbf{g}_{12}^{(j)}=|\mathbf{g}_{21}^{(j)}|,\  \mathbf{g}_{22}^{(j)}=-\frac{\mathbf{g}_{21}^{(j)}}{|\mathbf{g}_{21}^{(i)}|}(\mathbf{g}_{11}^{(j)})^*$, where $\mathbf{g}_{11}^{(j)}=\frac{\mathbf{h}_{jj}}{\mathbf{r}_{jj}},\ \mathbf{g}_{21}^{(j)}=\frac{\mathbf{h}_{j+1,j}}{\mathbf{r}_{jj}}$
\STATE Else
\STATE$\mathbf{g}_{22}^{(j)}=|\mathbf{g}_{11}^{(j)}|,\  \mathbf{g}_{12}^{(j)}=-\frac{\mathbf{g}_{11}^{(j)}}{|\mathbf{g}_{11}^{(j)}|}(\mathbf{g}_{21}^{(j)})^*$, where $\mathbf{g}_{11}^{(j)}=\frac{\mathbf{h}_{jj}}{\mathbf{r}_{jj}},\ \mathbf{g}_{21}^{(j)}=\frac{\mathbf{h}_{j+1,j}}{\mathbf{r}_{jj}}$
\STATE
\STATE $\mathbf{h}_{jj}=\|\left[\begin{array}{ccccc}
	\mathbf{h}_{jj}&
	\mathbf{h}_{j+1,j}\\
\end{array}\right]^T\|_2$, $\mathbf{h}_{j+1,j}=0$
\STATE $$\left[\begin{array}{ccccc}
	\mathbf{u}(j)\\
	\mathbf{u}(j+1)\\
\end{array}\right]=\left[\begin{array}{ccccc}
	\mathbf{g}_{11}^{(j)}&\mathbf{g}_{12}^{(j)}\\
	\mathbf{g}_{21}^{(j)}&\mathbf{g}_{22}^{(j)}\\
\end{array}\right]^*\left[\begin{array}{ccccc}
	\mathbf{u}(j)\\
	0\\
\end{array}\right]$$
\STATE If $|\mathbf{u}(j+1)|/\beta<\varepsilon$, then Stop
\STATE End Do
\STATE  Solve the quaternion linear systems $\widetilde{\mathbf{R}}_{j}\mathbf{y}_j=\mathbf{q}_1(1:j)$
  \STATE Compute the approximate solution $\mathbf{X}_j=\mathbf{X}_0+\mathbf{\mathcal{V}}_j*\mathbf{y}_j.$
	\end{algorithmic}
\end{breakablealgorithm}
\bigskip

Remarkably, when $m=1$, the proposed algorithm reduces to the QGMRES algorithm given by \cite{Jia4}. We also obtain that $\mathbf{h}_{j+1,j}=0$ if the algorithm stops at step $j$  with $\mathbf{B}-\mathbf{A}\mathbf{X}_j=\mathbf{0}.$ 
\begin{theorem}\label{tho-4-3}
Suppose that the quaternion matrix $\mathbf{A}$ is invertible. Then,
the Gl-QGMRES method breaks down at step $j$ if and only if the approximate solution $\mathbf{X}_j$ is exact.
\end{theorem}
\begin{proof}
    The proof is similar to Theorem 4.2 given by \cite{Jia4} and is omitted.
\end{proof}

Next, we show how to apply Gl-QFOM and Gl-QGMRES methods for solving the Sylvester quaternion matrix equation, but do not present any theoretical results.

Consider the well-known Sylvester quaternion matrix equation \begin{equation}\label{4-12}
\mathbf{A}\mathbf{X}+\mathbf{X}\mathbf{B}=\mathbf{C},\end{equation}
where $\mathbf{A}\in \mathbb{Q}^{ n\times n},\mathbf{B}\in \mathbb{Q}^{ m\times m}$ and $\mathbf{C}\in \mathbb{Q}^{ n\times m}$ are known matrices, $\mathbf{X}\in \mathbb{Q}^{ n\times m}$ is required to be determined. Such a matrix equation has widespread applications in system and control theory \cite{He2,He3}, neural network \cite{Zhang00}, and so on. Rodman \cite{Rodman} proved that the matrix equation has a unique solution if and only if the right eigenvalues of $\mathbf{A}$ and $\mathbf{B}$ are disjoint. 

Define
$$
\mathcal{A}(\mathbf{X}):=\mathbf{A}\mathbf{X}+\mathbf{X}\mathbf{B}.
$$
It is trivial to verify that the operator $\mathcal{A}$ is a liner mapping from the matrix space $\mathbb{Q}^{ n\times m}$ to itself. This shows that Eq.\eqref{1-1} can be rewritten as
\begin{equation}\label{4-13}
\mathcal{A}(\mathbf{X})=\mathbf{C}.
\end{equation}
By replacing the matrix-matrix product $\mathbf{A}\mathbf{X}$ with $\mathcal{A}(\mathbf{X})=\mathbf{A}\mathbf{X}+\mathbf{X}\mathbf{B}$, we then utilize the practical Gl-FOM and Gl-QGMRES methods to solve Eq.\eqref{4-12}. Let $\mathbf{R}_0$ be an initial guess and  $\mathbf{R}_0=\mathbf{B}-\mathbf{A}\mathbf{X}_0-\mathbf{X}_0\mathbf{B}$ be its corresponding residual. Here the quaternion matrix Krylov subspace $\mathcal{K}_j(\mathcal{A},\mathbf{R}_0)$ is defined as $$\mathcal{K}_j(\mathcal{A},\mathbf{R}_0):=\mbox{span}\{\mathbf{R}_0,\mathcal{A}(\mathbf{R}_0),\cdots,\mathcal{A}^{j-1}(\mathbf{R}_0)\},$$
and then $\mathbf{L}_j$ in Algorithm \ref{algorithm-5} is $\mathcal{A}\mathcal{K}_j(\mathcal{A},\mathbf{R}_0)$. Noting that $\mathcal{A}^{j-1}(\mathbf{R}_0)$ is defined recursively as $\mathcal{A}(\mathcal{A}^{j-2}(\mathbf{R}_0)).$ 

Next the practical Gl-FOM and Gl-QGMRES methods for solving the Sylvester quaternion matrix equation \eqref{4-12} are described below.

\bigskip
\begin{breakablealgorithm}
	\caption{ Practical Gl-QFOM method for solving Eq.\eqref{4-12}.}\label{algorithm-6}
	\begin{algorithmic}[1]
		 \STATE Choose an initial guess $\mathbf{X}_0$, and compute $\mathbf{R}_0=\mathbf{C}-\mathbf{A}\mathbf{X}_0-\mathbf{X}_0\mathbf{B}$, $\beta=\|\mathbf{R}_0\|$ and $\mathbf{V}_1=\mathbf{R}_0/\beta$, $\mathbf{u}=\beta\mathbf{e}_1$.
		\STATE 	For $j=1,2,\cdots$, construct an orthonormal basis of the Krylov subspace $\mathcal{K}_j(\mathcal{A},\mathbf{R}_0)$ by lines 3-7 of Algorithm \ref{algorithm-4}, and then solve the quaternion linear systems $\overline{\mathbf{R}}_{j}\mathbf{y}_j=\mathbf{u}$ by lines $8-21$ of Algorithm \ref{algorithm-4}.
  \STATE Compute the approximate solution $\mathbf{X}_j=\mathbf{X}_0+\mathbf{\mathcal{V}}_j*\mathbf{y}_j.$
	\end{algorithmic}
\end{breakablealgorithm}
\begin{breakablealgorithm}
	\caption{ Practical Gl-QGMRES method for solving Eq.\eqref{4-12}.}\label{algorithm-7}
	\begin{algorithmic}[1]
		 \STATE Choose an initial guess $\mathbf{X}_0$, and compute $\mathbf{R}_0=\mathbf{C}-\mathbf{A}\mathbf{X}_0-\mathbf{X}_0\mathbf{B}$, $\beta=\|\mathbf{R}_0\|$ and $\mathbf{V}_1=\mathbf{R}_0/\beta$, $\mathbf{u}=\beta\mathbf{\overline{e}}_1$.
		\STATE 	For $j=1,2,\cdots$, construct an orthonormal basis of the Krylov subspace $\mathcal{K}_j(\mathcal{A},\mathbf{R}_0)$ by line 3 of Algorithm \ref{algorithm-5}, and then solve the quaternion linear systems $\widetilde{\mathbf{R}}_{j}\mathbf{y}_j=\mathbf{q}_1(1:j)$ by lines $4-16$ of Algorithm \ref{algorithm-5}.
  \STATE Compute the approximate solution $\mathbf{X}_j=\mathbf{X}_0+\mathbf{\mathcal{V}}_j*\mathbf{y}_j.$
	\end{algorithmic}
\end{breakablealgorithm}
\bigskip

%The theoretical results and some preconditioners, including the Jacobi, Gauss-Seidel, SOR, and SSOR preconditioners, are under consideration for solving the original and Sylvester quaternion matrix equations.
%Let $\mbox{vec}(\mathbf{A})$ denote a vectorized matrix $\mathbf{A}$ by stacking its columns to form a vector. Then, by $\mbox{vec}(\mathbf{A}\mathbf{X}\mathbf{B})=(\mathbf{B}^*\otimes\mathbf{A})\mbox{vec}(\mathbf{X})$, Eq.\eqref{4-11} can be transformed into the following quaternion linear systems
%\begin{equation}\label{4-12}
%\over
%(\mathbf{I}\otimes\mathbf{A}+\mathbf{B}^*\otimes\mathbf{I})\mbox{vec}(\mathbf{X})=\mbox{vec}(\mathbf{C}),\end{equation}

\section{Numerical experiments}\label{section-5}

Far from being exhaustive, we in this section record some numerical results to illustrate the feasibility and superiority of the proposed algorithms when applied to solving Eq.\eqref{1-1} with random data and color image deblurring problems. All experiments were coded in Matlab (2022a) on a PC with Inter(R) Core(TM) i7-12700KF @3.6 GHz and 16.00 GB memory. For the sake of comparison, we are concerned with the proposed algorithms, Gl-GMRES, Gl-FOM \cite{Jbilou} QGMRES\cite{Jia4}, and QFOM\cite{Li0} methods in terms of the elapsed CPU time and iterations. For clarity, we denote the computing time in seconds by “CPU”, the iteration steps by “IT”, and
the relative residual $$\mathbf{RR}=\|\mathbf{B}-\mathbf{A}\mathbf{X}_k\|/\|\mathbf{R}_0\|.$$ Unless otherwise stated, the stopping
 criterion is set to be $\mathbf{RR}<\varepsilon=10^{-6}$,
or the maximum iteration does not exceed $k_{max}:= 3000$.  The symbol $\dagger$ means the relative residual did not reach the required accuracy after $3000$ iterations. In the following tests, the initial guess for all tested algorithms is a zero matrix.

\begin{example}\label{example-1}
{\rm Consider the original quaternion matrix equation \eqref{1-1} with its parameters given by
\begin{equation}\label{5-1}
\begin{aligned}
\mathbf{A}&=A_0+A_1\mathbf{i}+A_2\mathbf{j}+A_3\mathbf{k},\\
\mathbf{B}&=B_0+B_1\mathbf{i}+B_2\mathbf{j}+B_3\mathbf{k},
\end{aligned}
\end{equation}
where $B_i=\mbox{rand}(n,m)$, $i=1,2,3,4$, is an $n\times m$ random matrix with entries uniformly distributed in $[0,1]$, $A_1=-A_0$, $A_2=2*A_0$, $A_3=1.5*A_0$, and $A_0$ from the "MatrixMarket" \footnote{The MatrixMarket is available on https://math.nist.gov/MatrixMarket/} collection is chosen to be $A_0=\textbf{west0067}\ (n=67)$, $\textbf{bcspwr03}\ (n=118)$ and $\textbf{bcsstk05}\ (n=153)$, respectively.}

{\rm  Then we implemented Algorithms \ref{algorithm-4}, \ref{algorithm-5}, Gl-GMRES and Gl-FOM methods using real counterparts, for solving Eq.\eqref{1-1}. The obtained numerical results were reported in Table \ref{Table-1}. As shown in this table, one can see that Algorithms \ref{algorithm-4} and \ref{algorithm-5} both work more effectively than Gl-FOM and Gl-GMRES. Specifically, the number of iteration steps and the elapsed CPU time corresponding to the former two are much less than the latter two. The reason is that the arithmetic operations of Gl-FOM and Gl-GMRES methods for solving Eq.\eqref{1-1} are about four times to Algorithms \ref{algorithm-4} and \ref{algorithm-5} in each iteration step, which is very time-consuming.  In case $n=67,m=3$, the convergence curves of all algorithms were depicted in Figure \ref{fig:1}. The curves in this Figure show that, the convergence behaviour of Algorithm \ref{algorithm-5} performs more reliably than Algorithm \ref{algorithm-4}, even though the computing time to the former is slightly less than the latter.
}
\end{example}
 
\begin{table}[htbp]
	\footnotesize
	\caption{Numerical results of Example \ref{example-1}. \label{Table-1}}
 \renewcommand\arraystretch{1}
	\begin{center}
		\begin{tabular}{|c|c|c|c|c|c|} \hline
			\bf Case &\bf Algorithm &\bf Dimension& \bf IT&\bf CPU&\bf RR \\ 
   &&$[n,m]$ &&& \\ \hline
   &Gl-FOM&$[268,12]$&307&0.9789& 4.8776e-07\\
	\textbf{west0067}&Gl-GMRES& $[268,12]$&307& 0.9894& 4.4517e-07\\ 
			& Algorithm \ref{algorithm-4} &$[67,3]$&184&0.3613&3.5667e-07\\
			& Algorithm \ref{algorithm-5} &$[67,3]$&184&0.3725&  3.4831e-07\\
			\hline
&Gl-FOM&$[472,20]$&780&21.5197&8.3133e-07\\
		\textbf{bcspwr03} &Gl-GMRES& $[472,20]$&781&21.4393& 6.6087e-07\\ 
		& Algorithm \ref{algorithm-4} &$[118,5]$&479&4.8358& 9.7107e-07\\
			& Algorithm \ref{algorithm-5} &$[118,5]$&480&4.9553& 9.7908e-07\\
			\hline
   &Gl-FOM&$[612,36]$&2820&511.1095&8.0738e-07\\
	\textbf{bcsstk05}&Gl-GMRES& $[612,36]$&2821&512.8517& 7.4933e-07\\ 
			 & Algorithm \ref{algorithm-4} &$[153,9]$&985&31.2043&4.4118e-07 \\
			& Algorithm \ref{algorithm-5} &$[153,9]$&986&31.8455&9.3870e-07\\
   \hline
		\end{tabular}
	\end{center}
\end{table}

\begin{figure}[!htbp]
	\hspace{-6em}
	\begin{center}
		\begin{minipage}[c]{0.4\textwidth}
			\includegraphics[width=3in]{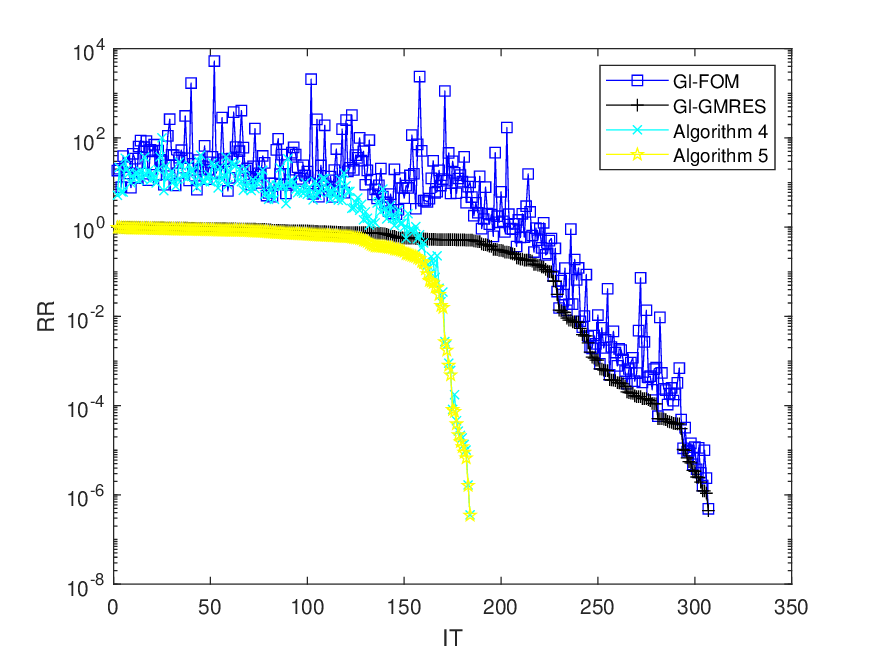}
		\end{minipage}
	\end{center}
	\vspace{-1em}\caption{ Convergence curves of Example \ref{example-1} with $n=67, m=3$.\label{fig:1}}
\end{figure}

\begin{table}[htbp]
	\footnotesize
	\caption{Numerical results of Example \ref{example-2}. \label{Table-2}}
 \renewcommand\arraystretch{1}
	\begin{center}
		\begin{tabular}{|c|c|c|c|c|c|} \hline
			\bf Case &\bf Algorithm &\bf Dimension& \bf IT&\bf CPU&\bf RR \\ 
   &&$[n,m]$ &&& \\ \hline
   &Gl-FOM&$[128,20]$&973& 6.6117& 9.5987e-07\\
	\textbf{ibm32}&Gl-GMRES& $[128,20]$&935&  6.0750& 9.3089e-07\\ 	
 & Algorithm \ref{algorithm-4} &$[32,5]$&127&0.3087&7.4128e-07\\
& Algorithm \ref{algorithm-5} &$[32,5]$&128& 0.3162&  6.8533e-08\\ \hline
 &Gl-FOM&$[340,20]$&2380& 116.4435&9.6364e-07\\
\textbf{ash85}&Gl-GMRES& $[340,20]$&2235& 109.3513& 9.9626e-07\\ 
& Algorithm \ref{algorithm-4} &$[85,5]$&294&14.3012&2.2206e-07\\
			& Algorithm \ref{algorithm-5} &$[85,5]$&295& 14.6479&  8.4313e-07\\ \hline
    &Gl-FOM&$[900,60]$&$\dagger$&$\dagger$&$\dagger$\\%584.0237
	\textbf{pde225}&Gl-GMRES& $[900,60]$&$\dagger$&$\dagger$& $\dagger$\\ 
 & Algorithm \ref{algorithm-4} &$[225,15]$&804&44.3952&9.4365e-07\\
& Algorithm \ref{algorithm-5} &$[225,15]$&804&  44.7286& 9.4973e-07
\\ \hline
  &Gl-FOM&$[1780,80]$&$\dagger$&$\dagger$&$\dagger$\\%584.0237
	\textbf{can445}&Gl-GMRES& $[1780,80]$&$\dagger$&$\dagger$& $\dagger$\\ 
			& Algorithm \ref{algorithm-4} &$[445,20]$&2736&514.6701&9.2975e-07\\
			& Algorithm \ref{algorithm-5} &$[445,20]$&2737&  515.8727&9.9578e-07
\\ \hline
		\end{tabular}
	\end{center}
\end{table}

\begin{example}\label{example-2}
{\rm Consider the Sylvester quaternion matrix equation \eqref{4-12} with its coefficient matrices $\mathbf{A}$ and $\mathbf{B}$ being defined as Eq.\eqref{5-1}. Here the real matrices $A_i$ and $B_i$ are chosen to be
$$\begin{aligned}
 a).\ A_0&=\textbf{ibm32}\ (n=32), B_1=2*B_0,\ B_2=-B_0,\ B_3=1.5*B_0,\\ 
 b).\ A_0&=\textbf{ash85}\ (n=85), B_1=B_0,\ B_2=-B_0,\ B_3=-1.5*B_0,\\
 c).\ A_0&=\textbf{pde225}\ (n=225), B_1=B_0,\ B_2=2*B_0,\ B_3=-B_0,\\
d).\ A_0&=\textbf{can445}\ (n=445), B_1=1.5*B_0,\ B_2=B_0,\ B_3=-B_0,\\
\end{aligned}$$
with $$B_0=\left[\begin{array}{ccccccccc}
	2&1&&&\\
	-1&2&1&&\\
   &\ddots&\ddots&\ddots&\\
   &&-1&2&1\\
   &&&-1&2\\
\end{array}\right]$$
of order $m$, and $C_i=\mbox{rand}(n,m)$, $i=1,2,3,4$, is a random matrix of size $n\times m.$
}
\end{example}

{\rm Applying Gl-FOM, Gl-GMRES, Gl-QFOM, and Gl-QGMRES for solving  Eq.\eqref{4-12}, we obtained the numerical results listed in Table \ref{Table-2}. We can see from the table that, all of the algorithms for solving Eq.\eqref{4-12} are feasible except for Gl-FOM and Gl-GMRES with cases \textbf{pde225}, \textbf{can445}, while
Algorithms \ref{algorithm-4} and \ref{algorithm-5} converge in fewer iterations and cost less the computing time than other solvers. Additionally, the performance of Algorithm \ref{algorithm-4} is slightly better than that of Algorithm \ref{algorithm-5}.}

In what follows, we test the effectiveness of the proposed algorithms for solving model \eqref{1-1} applied in color image restoration.  The original color image, denoted by a pure quaternion matrix $\mathbf{X}^*=X_1\mathbf{i}+X_2\mathbf{j}+X_3\mathbf{k}\in \mathbb{Q}^{m\times n}$, consists of $n\times m$ color image pixel values in the range $[0, d]$, where $d=255$ represents the maximum possible pixel value of the image, and $X_1,X_2,X_3\in \mathbb{R}^{n\times m}$ denote the red, green and blue channels, respectively.
The quaternion matrix $\mathbf{B}=\mathbf{A}\mathbf{X}$, i.e., Eq.\eqref{1-1},
denotes the blurred and noise-free image, where $\mathbf{A}\in \mathbb{Q}^{n\times n}$ is the blurring matrix. Here we compare numerically, our algorithms with the QFOM \cite{Li0} and QGMRES \cite{Jia4} methods since they outperform many existing methods in terms of the iterations and the computing time. In fact, for Eq.\eqref{1-1} applied in color image restoration, the QFOM \cite{Li0} and QGMRES always stack the columns of $\mathbf{X}$, such that \eqref{1-1} can be equivalently transformed into the following large-scale quaternion linear systems
\begin{equation}\label{5-2}
\hat{\mathbf{A}}\hat{\mathbf{x}}=\hat{\mathbf{b}},
\end{equation}
where $$\hat{\mathbf{A}}=\left[\begin{array}{ccccccccc}
	\mathbf{A}&&&\\
	&\mathbf{A}&&\\
   &&\ddots&\\
   &&&\mathbf{A}\\
\end{array}\right]$$
is a block diagonal matrix of size $n^2\times n^2$, $\hat{\mathbf{x}}\in \mathbb{Q}^{nm}$ and $\hat{\mathbf{b}}\in \mathbb{Q}^{nm}$ are quaternion vectors obtained by stacking the columns of $\mathbf{X}$ and $\mathbf{B}$. Clearly, the QFOM and QGMRES methods for solving Eq.\eqref{5-2} are very costly compared to Algorithms \ref{algorithm-4} and \ref{algorithm-5} for solving  Eq.\eqref{1-1}. The reason is that the dimension of Eq.\eqref{5-2} is quite large, while the proposed algorithms for solving  Eq.\eqref{1-1} save three-quarters of arithmetic operations and storage space.

The performance of all tested algorithms is evaluated by four measures the computing time ($\mathbf{CPU}$), the peak signal-to-noise ratio ($\mathbf{PSNR}$) in decibel ($\mathbf{dB}$), structural similarity ($\mathbf{SSIM}$) \cite{Wang0101}, relative residual ($\mathbf{RR}$). Define
$$\begin{aligned}
\mathbf{PSNR}(\mathbf{X})&=10\log_{10}(\frac{3mnd^2}{\|\mathbf{X}-\mathbf{X}_k\|^2}),\quad \mathbf{RR}(\mathbf{X})=\frac{\|\mathbf{X}-\mathbf{X}_k\|}{\|\mathbf{X}\|},  \\
\mathbf{SSIM}(\mathbf{X})&=\frac{(2\mu_{\mathbf{X}}\mu_{\mathbf{X}_k}+c_1)(2\sigma_{\mathbf{X}\mathbf{X}_k}+c_2)}{(\mu_{\mathbf{X}}^2+\mu_{\mathbf{X}_k}^2+c_1)(\sigma_{\mathbf{X}}^2+\sigma_{\mathbf{X}_k}^2+c_2)},
\end{aligned}
$$
where $\mu_{(\mathbf{X})}$, $\mu_{(\mathbf{X}_k)}$ and $\sigma_{(\mathbf{X})}$, $\sigma_{(\mathbf{X})_k}$  are the means and variances of  $\mathbf{X}$ and $\mathbf{X}_k$, respectively, $\sigma_{\mathbf{X}\mathbf{X}_k}$ is the covariance between $\mathbf{X}$ and $\mathbf{X}_k$, $c_1=(0.01L)^2, c_2=(0.03L)^2$ with $L$ being the dynamic range of pixel values. Here the stopping
 criterion is $\mathbf{RR}<\varepsilon=10^{-2}$.

\begin{example}\label{example-3}
{\rm In this example, we use two sets of experiments. Consider three original color images, including the "\textbf{Lanhua}", "\textbf{Hongyi}", "\textbf{Peppers}" color images of size $128\times 128$, blurred by two kinds of single channel blurring matrices $\mathbf{A}=\mathbf{A}_0\in \mathbb{R}^{n\times n}$. }

{\rm In the first experiment, the blurring matrix $\mathbf{A}_0=(a_{ij})\in \mathbb{R}^{128\times 128}$ is the Toeplitz matrix \cite{Jbilou0,Jbilou1} with its entries given by
\begin{equation}\label{5-3}
a_{ij}=\left\{\begin{aligned}	
	&\frac{1}{2s-1},\ \ |i-j|\leq s,\\
	&0,\ \ \ \  \mbox{otherwise},
\end{aligned}
\right.
\end{equation}
which models a uniform blur. The parameter $s$ is chosen to be $20$. We implemented the QFOM, QGMRES and the proposed algorithms for solving Eqs.\eqref{5-2} and \eqref{1-1}, respectively, and then reported the numerical results and the restored images in Table \ref{Table-3} and Figure \ref{fig:2}. From this table, Algorithms \ref{algorithm-4} and \ref{algorithm-5} can restore images with higher quality than other solvers and require less computing time, which are up-and-coming methods. Moreover, we can see that the elapsed computing time, PSNR, and SSIM corresponding to Algorithm \ref{algorithm-4} is slightly better than Algorithm \ref{algorithm-5}.
\begin{table}[htbp]
	\footnotesize
	\caption{Numerical results of Example \ref{example-3} with uniform blur. \label{Table-3}}
 \renewcommand\arraystretch{1}
	\begin{center}
		\begin{tabular}{|c|c|c|c|c|c|} \hline
			\bf Case &\bf Algorithm &\bf PSNR & \bf SSIM &\bf CPU&\bf RR \\  \hline
   &QFOM&32.5307&0.9304&244.7376&4.7165e-02\\
\textbf{Lanhua}&QGMRES&32.5185&0.9254&245.6593&4.8013e-02\\ 
   &Algorithm \ref{algorithm-4} &36.0217&0.9728&18.8159& 3.1795e-02\\
	& Algorithm \ref{algorithm-5}&36.0186&0.9716&  19.1750& 3.1842e-02\\ 	
 \hline
 &QFOM&30.8243&0.9454&235.7022&4.3156e-02\\
\textbf{Hongyi}&QGMRES&30.8189&0.9430&236.4835&4.3213e-02\\ 
   &Algorithm \ref{algorithm-4} &32.3657&0.9774&17.8476& 2.4218e-02\\
	& Algorithm \ref{algorithm-5}&32.3396&0.9729&  18.3281& 2.4637e-02\\ 	
 \hline
 &QFOM&33.9412&0.9389&256.3754&4.2516e-02\\
\textbf{Peppers}&QGMRES&33.9275&0.9327&257.4598&4.2749e-02\\ 
   &Algorithm \ref{algorithm-4} &37.6389&0.9761&19.9556& 2.6186e-02\\
	& Algorithm \ref{algorithm-5}&37.6324&0.9712&  20.0189& 2.6237e-02\\ 	
 \hline
		\end{tabular}
	\end{center}
\end{table}
\begin{figure}[!htbp]
	\hspace{-6em}
	\begin{center}
		\begin{minipage}[c]{1.2\textwidth}
			\includegraphics[width=6in]{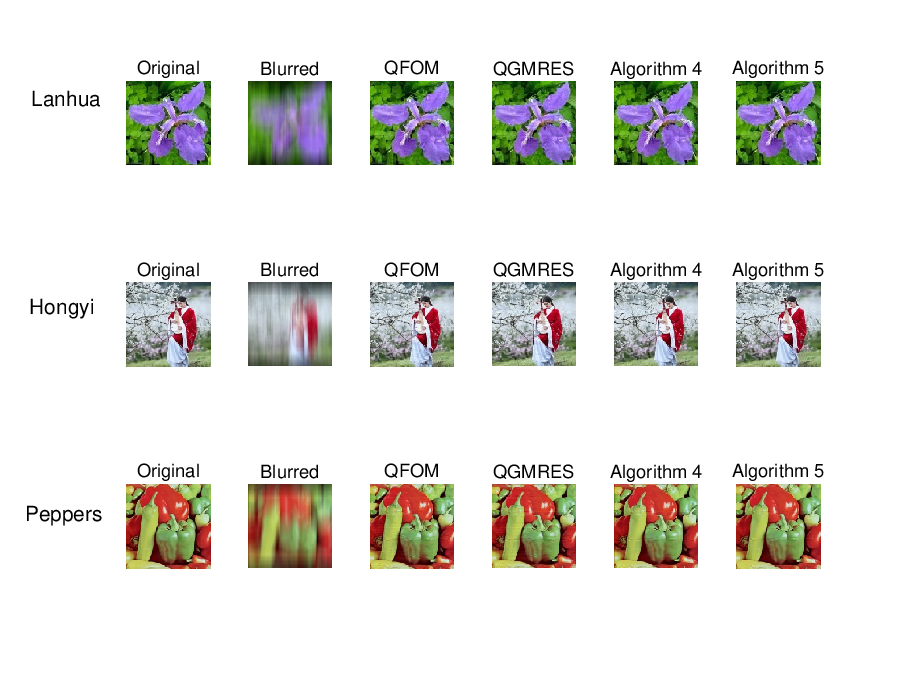}
		\end{minipage}
	\end{center}
	\vspace{-5em}\caption{Left to Right: the original images of size $128\times 128$, the observed images with uniform blur, the deblurred images obtained by QFOM and QGMRES, Algorithms \ref{algorithm-4} and \ref{algorithm-5}, respectively.\label{fig:2}}
\end{figure}

In the second experiment,  the blurring matrix $\mathbf{A}_0\in \mathbb{R}^{128\times 128}$ is also the Toeplitz matrix \cite{Jbilou2,Li1212} with its entries given by
\begin{equation}\label{5-4}
a_{ij}=\left\{\begin{aligned}	
&\frac{1}{\sigma\sqrt{2\pi}}\exp(-\frac{(i-j)^2}{2\sigma^2}),\ \ |i-j|\leq r,\\
&0,\ \ \ \  \mbox{otherwise},
\end{aligned}
\right.
\end{equation}
where the parameters $r=35$ and $\sigma=10$. It models a blur arising in connection with the degradation of digital images by atmospheric turbulence blur. Similarly, we ran all the tested algorithms for solving Eqs.\eqref{5-2} and \eqref{1-1}. The obtained numerical results were recorded in Table \ref{Table-4}, and the restored colour images by the algorithms were shown in Figure \ref{fig:3}. As seen, the proposed algorithms still outperform QFOM and QGMRES regarding the elapsed computing time, PSNR and SSIM.}
\end{example}

\begin{table}[htbp]
	\footnotesize
	\caption{Numerical results of Example \ref{example-3} with atmospheric
turbulence blur. \label{Table-4}}
 \renewcommand\arraystretch{1}
	\begin{center}
		\begin{tabular}{|c|c|c|c|c|c|} \hline
			\bf Case &\bf Algorithm &\bf PSNR & \bf SSIM &\bf CPU&\bf RR \\  \hline
   &QFOM&27.1721&0.8752&274.4323&8.8278e-02\\
  \textbf{Lanhua} &QGMRES&27.1646&0.8714&275.0952&8.8356e-02\\ 
   &Algorithm \ref{algorithm-4}  &30.8217&0.9412&21.6829&5.6793e-02\\ 
   &Algorithm \ref{algorithm-5} &30.8135&0.9397&22.1847&5.7725e-02\\  \hline
     &QFOM&26.7325&0.8937&580.5726&7.1253e-02\\ 
  \textbf{Hongyi} &QGMRES&26.7194&0.8925&581.5637&7.2568e-02\\ 
   &Algorithm \ref{algorithm-4} &30.2372&0.9542&36.1186&4.7273e-02\\ 
   &Algorithm \ref{algorithm-5} &30.2154&0.9516&36.4924&4.7925e-02\\ \hline
     &QFOM&27.7925&0.8847&282.3236&9.1373e-02\\ 
  \textbf{Peppers} &QGMRES&27.7428&0.8749&283.3960&9.1128e-02\\ 
   &Algorithm \ref{algorithm-4} &31.1406&0.9423&23.1272&5.4579e-02\\ 
   &Algorithm \ref{algorithm-5} &31.1179&0.9401&23.9685&5.4154e-02\\ \hline
		\end{tabular}
	\end{center}
\end{table}
\begin{figure}[!htbp]
	\hspace{-6em}
	\begin{center}
		\begin{minipage}[c]{1.2\textwidth}
			\includegraphics[width=6in]{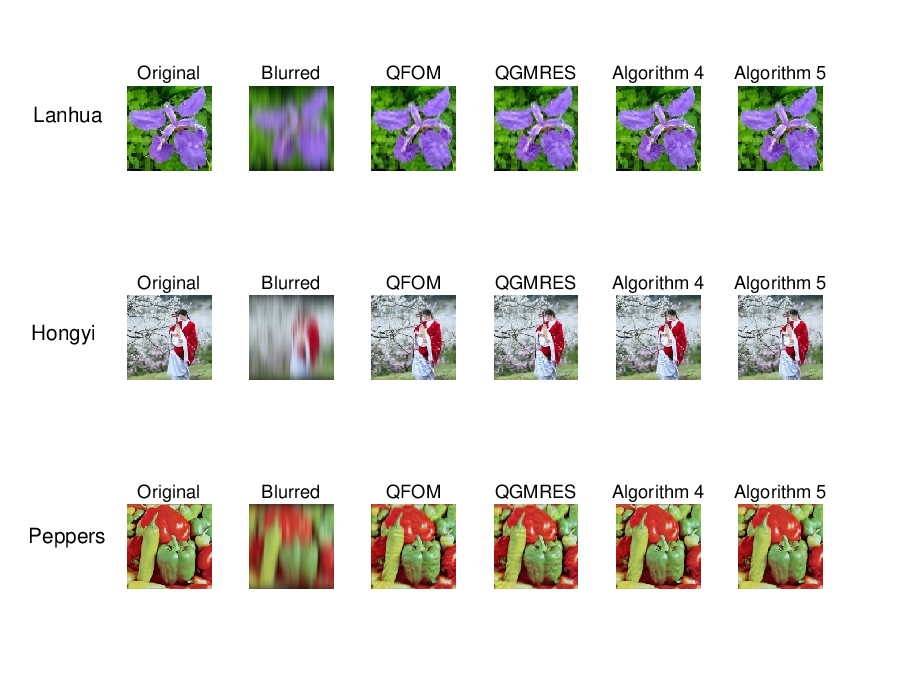}
		\end{minipage}
	\end{center}
	\vspace{-5em}\caption{ Left to Right: the original images of size $128\times 128$, the observed images with atmospheric turbulence blur, the deblurred images obtained by QFOM and QGMRES, Algorithms \ref{algorithm-4} and \ref{algorithm-5}, respectively.\label{fig:3}}
\end{figure}

\begin{table}[htbp]
	\footnotesize
	\caption{Numerical results of Example \ref{example-4} with Multichannel. \label{Table-5}}
 \renewcommand\arraystretch{1}
	\begin{center}
		\begin{tabular}{|c|c|c|c|c|c|} \hline
			\bf Case &\bf Algorithm &\bf PSNR & \bf SSIM &\bf CPU&\bf RR \\  \hline
   &QFOM&28.2987&0.8852&305.7253&7.3693e-02\\
  \textbf{Lena} &QGMRES&28.2869&0.8814&306.5478&7.3869e-02\\ 
   &Algorithm \ref{algorithm-4}  &29.1257&0.9284&27.1429&6.7995e-02\\ 
   &Algorithm \ref{algorithm-5} &29.1168&0.9238&27.5369&6.8761e-02\\  \hline
   &QFOM&26.3089&0.8662&472.3548&9.5471e-02\\
  \textbf{Peppers} &QGMRES&26.2946&0.8628&473.2271&9.5638e-02\\ 
   &Algorithm \ref{algorithm-4}  &27.1036&0.9115&35.9783&8.6315e-02\\ 
   &Algorithm \ref{algorithm-5} &27.0954&0.9102&36.3629&8.6737e-02\\  \hline
		\end{tabular}
	\end{center}
\end{table}
\begin{figure}[!htbp]
	\hspace{-6em}
	\begin{center}
		\begin{minipage}[c]{1\textwidth}
			\includegraphics[width=6in]{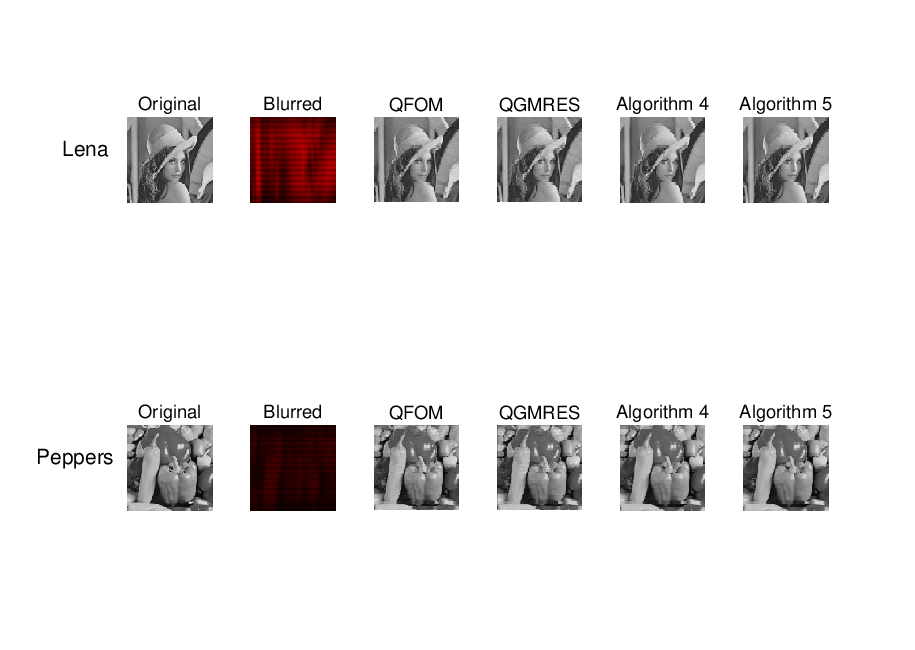}
		\end{minipage}
	\end{center}
	\vspace{-5em}\caption{  Left to Right: the original images of size $128\times 128$, the observed images with multichannel blur, the deblurred images obtained by QFOM, QGMRES, Algorithms \ref{algorithm-4} and \ref{algorithm-5}, respectively.\label{fig:4}}
\end{figure}
\begin{example}\label{example-4}
{\rm In this example, we consider two original colour images, including the "\textbf{Lena}", "\textbf{Peppers}" colour images of size $128\times 128$, blurred by a multichannel blurring matrix $\mathbf{A}=\mathbf{A}_1\mathbf{i}+\mathbf{A}_2\mathbf{j}+\mathbf{A}_3\mathbf{k}\in \mathbb{Q}^{n\times n}$. Note that the images mentioned above are colour images with the same red, green, and blue channels, despite seeming like grey-scale images. Let $\mathbf{H}_0\in\mathbb{R}^{16\times 16}$ defined in \eqref{5-4} be the Toeplitz matrix with $r=\sigma=3$, and let $\mathbf{H}_1\in\mathbb{R}^{8\times 8}$ defined in \eqref{5-3} be the Toeplitz matrix with $s=5$. Define $\mathbf{A}_1=\mathbf{H}_0\otimes \mathbf{H}_1\in \mathbb{R}^{128\times 128}$, $\mathbf{A}_2=\mathbf{A}_3=-0.5\mathbf{A}_1$. 
}

{\rm We applied the QFOM, QGREMS, and Algorithms \ref{algorithm-4} and \ref{algorithm-5} to deal with the restoration problem. We reported the numerical results and
the restored colour images by the above algorithms in Table \ref{Table-5} and Figure \ref{fig:4}, respectively. We see from Table \ref{Table-5} that the restoration quality by Algorithms \ref{algorithm-4} and \ref{algorithm-5} is still better than that by  QFOM and QGREMS in terms of PSNR and SSIM, especially in the computing time.}
\end{example}

\section{Conclusion}\label{section-6}
In this paper, we proposed the Gl-QFOM and Gl-GMRES methods for solving quaternion linear systems with multiple
right-hand sides. The main contributions are listed as follows.
 \begin{itemize}
\item The global quaternion Arnoldi method, which inherits the algebraic symmetry of the real counterpart, is first proposed. On this basis, two structure-preserving methods for overcoming the multiplicative noncommutativity of quaternions, including the Gl-QFOM and Gl-GMRES, are developed for solving quaternion linear systems with multiple right-hand sides.
\item Gl-QFOM and Gl-QGMRES are novel methods presented for solving the Sylvester tensor equation over the quaternion algebra, and outperform some existing solvers in terms of the computing time and iterations.
\item Gl-QFOM and Gl-QGMRES are successfully applied to the color image processing and require much less time than QFOM and QGMRES.
\end{itemize}
In the future, we are interested in developing the preconditioned Gl-QFOM and Gl-QGMRES for other applications.

\bigskip

\textbf{Contributions:} {All authors have equal contributions for Conceptualization, Formal analysis, 
Investigation, Methodology, Software, Validation, Writing an original draft, Writing a review, and 
Editing. All authors have read and agreed to the published version of the manuscript.}

\bigskip
\textbf{Data Availability Statement:} {The data used to support the findings of this study are included within this article.}
\bigskip

\textbf{Conflict of Interest:} {The authors declare no conflict of interest.}
%\section*{Acknowledgment}


\begin{thebibliography}{99}



\bibitem{Hamilton} W.R. Hamilton, Elements of Quaternions, \textit{Longmans Green and Co London}, 1866.

\bibitem{Zhang1} F.Z. Zhang, Quaternions and matrices of quaternions, \textit{Linear Algebra Appl.}, 251 (1997) 21-57.

\bibitem{Rodman}L. Rodman, Topics in Quaternion Linear Algebra, \textit{Princeton Series in Applied Mathematics}, Princeton University Press, Princeton, NJ, 2014.


\bibitem{Dyson} F.J. Dyson, Statistical theory of the energy levels of complex systems, \textit{J. Math. Phys.}, 3 (1962) 140-156.

\bibitem{Ginzberg} P. Ginzberg, Quaternion matrices: statistical properties and applications to signal processing and wavelets, \textit{Imperial College London}, 2013.

\bibitem{Adler} S.L. Adler, Quaternionic quantum mechanics and quantum fields, \textit{Oxford U.P.}, New York, 1994. 

\bibitem{Chu0} Y. Guan, M.T. Chu, D.L. Chu, SVD-based algorithms for the best rank-1 approximation of a symmetric tensor, \textit{SIAM J. Matrix Anal. Appl.}, 39 (2018) 1095-1115.

\bibitem{Chu1} Y. Guan, D.L. Chu, Numerical computation for orthogonal low-rank approximation of tensors, \textit{SIAM J. Matrix Anal. Appl.}, 40 (2019) 1047-1065.

\bibitem{Chu2} D.L. Chu, W. Y. Shi, S. Eswar, An alternating rank-$k$ nonnegative least squares framework (ARkNLS) for nonnegative matrix factorization, \textit{SIAM J. Matrix Anal. Appl.}, 42 (2021) 1451-1479.

\bibitem{Wang1}  Q.W. Wang, The general solution to a system of real quaternion matrix equations, \textit{Comput. Math. Appl.}, 49 (2005) 665-675.

\bibitem{He0} Z.H. He, Q.W. Wang, A real quaternion matrix equation with applications, \textit{Linear and Multilinear Algebra},  61 (2013) 725-740.

\bibitem{Yuan} S.F. Yuan,  Q.W. Wang,  X.F. Duan, On solutions of the quaternion matrix equation $AX= B$ and their applications in color image restoration, \textit{Appl. Math. Comput.} 221 (2013) 10-20.

\bibitem{Jia0} Z.G. Jia, M.K, Ng, G.J. Song, Robust quaternion matrix completion with applications to image inpainting, \textit{Numer. Linear Algebra Appl.}, 26 (2019) e2245.

\bibitem{He1} Z.H. He, W.L. Qin, X.X. Wang, Some applications of a decomposition for five quaternion matrices in control system and color image processing, \textit{Comput. Appl. Math.},  40 (2021)  1-29.

\bibitem{Chen00} J.F.  Chen,  Q.W. Wang,  G.J. Song, Quaternion matrix factorization for low-rank quaternion matrix completion, \textit{Mathematics}, 11 (2023) 2144.

\bibitem{Ji} T.Y. Ji, D.L. Chu, X.L. Zhao, A unified framework of cloud detection and removal based on low-rank and group sparse regularizations for multitemporal multispectral images, \textit{IEEE Trans. Geosci. Remote. Sens}, 60 (2022) 1-15.

\bibitem{Leary} D. O’Leary, The block conjugate gradient algorithm and related methods, \textit{Linear Algebra Appl.}, 29 (1980) 293–322.

\bibitem{Simoncini} V. Simoncini, E. Gallopoulos, Convergence properties of block GMRES and matrix polynomial, \textit{Linear
Algebra Appl.}, 247 (1996) 97-119.

\bibitem{Freund} R. Freund, M. Malhotra, A block-QMR algorithm for non-Hermitian linear systems with multiple right-hand sides, \textit{Linear Algebra Appl.}, 254 (1997) 119-157.

\bibitem{Jbilou}  K. Jbilou, A. Messaoudi, H. Sadok, Global FOM and GMRES algorithms for matrix equations, \textit{Appl. Numer. Math.},  31 (1999) 49-63.

\bibitem{Bouyouli}  R. Bouyouli, K. Jbilou, R. Sadaka, Convergence properties of some block Krylov subspace methods for multiple linear systems, \textit{J. Comput. Appl. Math.},  196 (2006) 498-511.

\bibitem{Jia4} Z.G. Jia, M.K. Ng, Structure preserving quaternion generalized minimal residual method, \textit{SIAM J. Matrix Anal. Appl.}, 42 (2021) 616-634.

\bibitem{Chen0} X. Chen, Quaternion generalized minimal residual method with applications to image processing, 2019.

\bibitem{Li0} T. Li,  Q.W. Wang, Structure preserving quaternion full orthogonalization method with applications, \textit{Numer. Linear Algebra Appl.}, (2023) e2495.

\bibitem{Wei0} M.S. Wei, Y. Li, Theory and calculation of generalized least square problem, \textit{Science Press}, 9 (2006) 107-115.

\bibitem{Jia1} Z.G. Jia, The eigenvalue problem of quaternion matrix: structure-preserving algorithms and applications, \textit{Science Press}, Beijing, 2019.

\bibitem{Jia2} Z.G. Jia, M.S. Wei, S.T. Ling, A new structure-preserving method for quaternion Hermitian eigenvalue problems, \textit{J. Comput. Appl. Math.}, 239 (2013) 12-24.

\bibitem{Jia3} Z.G. Jia, M.S. Wei, M.X. Zhao, Y. Chen, A new real structure-preserving quaternion
QR algorithm, \textit{J. Comput. Appl. Math.}, 343 (2018) 26-48.


\bibitem{Ghiloni} R. Ghiloni, V. Moretti, A. Perotti, Continous slice functional calculus in quaternionic Hilbert spaces, \textit{Rev. Math. Phys.}, 25 (2013) 1350006.

\bibitem{Saad1} Y. Saad, Iterative methods for sparse linear systems, \textit{SIAM}, 2003.

\bibitem{He2} Z.H. He, Q.W.  Wang, Y. Zhang, A system of quaternary coupled Sylvester-type real quaternion matrix equations, \textit{Automatica}, 87 (2018) 25-31.

\bibitem{He3} Q.W.  Wang, Z.H. He, Y. Zhang, Constrained two-sided coupled Sylvester-type quaternion matrix equations, \textit{Automatica}, 101 (2019) 207-213.


\bibitem{Zhang00} Y.N. Zhang, D.C. Jiang, J. Wang, A recurrent neural network for solving Sylvester equation with time-varying coefficients,
\textit{IEEE Trans. Neural Networks}, 13 (2002) 1053-1063.

\bibitem{Wang0101} Z. Wang, A.C. Bovik, H.R. Sheikh, Image quality assessment: from error visibility to structural similarity, \textit{IEEE Trans. Image Process.}, 13 (2004) 600–612.

\bibitem{Jbilou0} A. Bouhamidi, K. Jbilou, Sylvester Tikhonov-regularization methods in image restoration, \textit{J. Comput. Appl. Math.}, 206 (2007) 86-98.

\bibitem{Jbilou1} A. Bouhamidi, K. Jbilou, A note on the numerical approximate solutions for generalized Sylvester matrix equations with applications, \textit{J. Comput. Appl. Math.}, 206 (2008) 687-694.

\bibitem{Jbilou2}  A. Bouhamidi, R. Enkhbat,  K. Jbilou, Conditional gradient Tikhonov method for a convex optimization problem in image restoration, \textit{J. Comput. Appl. Math.}, 255 (2014) 580-592.

\bibitem{Li1212}J.F. Li, W. Li, R. Huang, An efficient method for solving a matrix least squares problem over a matrix inequality constraint, \textit{Comput. Optim. Appl.}, 63 (2) (2016) 393-423.


\end{thebibliography}
\end{document}